\documentclass[11pt]{amsart}
\usepackage{amsthm}
\usepackage{amssymb}
\usepackage{amsmath}
\usepackage{tikz}
\usepackage{url}
\usepackage{cite}
\nocite{*}
\usepackage{a4wide}

\newtheorem{thm}{Theorem}

\newtheorem{lem}[thm]{Lemma}

\theoremstyle{remark}
\newtheorem{rem}{Remark}

\newcommand{\comment}[1]{}

\DeclareMathOperator{\I}{I}
\DeclareMathOperator{\J}{J}
\DeclareMathOperator{\rr}{r}
\DeclareMathOperator{\h}{h}
\DeclareMathOperator{\dd}{d}

\newcommand{\ics}{infima closed set}
\newcommand{\icss}{infima closed sets}

\makeatletter
\def\Ddots{\mathinner{\mkern1mu\raise\p@
\vbox{\kern7\p@\hbox{.}}\mkern2mu
\raise4\p@\hbox{.}\mkern2mu\raise7\p@\hbox{.}\mkern1mu}}
\makeatother
\begin{document}

\thispagestyle{plain}

\title{Trees with minimum number of infima closed sets}

\author{Eric Ould Dadah Andriantiana}
\address{Eric Ould Dadah Andriantiana\\
Department of Mathematics (Pure and Applied)\\
Rhodes University, PO Box 94\\
6140 Grahamstown\\
South Africa}
\email{E.Andriantiana@ru.ac.za}

\author{Stephan Wagner}
\address{Stephan Wagner\\
Department of Mathematics \\
Uppsala Universitet \\
Box 480 \\
751 06 Uppsala \\
Sweden
\and
Department of Mathematical Sciences\\
Stellenbosch University\\
Private Bag X1\\
Matieland 7602\\
South Africa
}
\email{stephan.wagner@math.uu.se,swagner@sun.ac.za}
\thanks{This work was supported by grants from the National Research Foundation of South Africa (grants 96236 and 96310). The second author was supported by the Knut and Alice Wallenberg Foundation.}
\date{\today}

\begin{abstract}
Let $T$ be a rooted tree, and $V(T)$ its set of vertices. A subset $X$ of $V(T)$ is called an \ics{} of $T$ if for any two vertices $u,v\in X$, the first common ancestor of $u$ and $v$ is also in $X$. This paper determines the trees with minimum number of \icss{} among all rooted trees of given order, thereby answering a question of Klazar. It is shown that these trees are essentially complete binary trees, with the exception of vertices at the last levels. Moreover, an asymptotic estimate for the minimum number of \icss{} in a tree with $n$ vertices is also provided.
\end{abstract}

\maketitle

\section{Introduction}
Let $T$ be a rooted tree, and let the root be denoted by $\rr(T)$. 
The vertices of $T$ can be regarded as the elements of a poset, with $\rr(T)$ as its least element. Vertices $v$ and $w$ satisfy the relation $v \preceq w$ if and only if $v$ lies on the path from the root $\rr(T)$ to $w$. The tree $T$ can then be regarded as the corresponding Hasse diagram. The infimum of two vertices $u$ and $v$ in the sense of this poset is the common vertex of the paths from $u$ to the root and from $v$ to the root that is furthest from the root. A subset $X$ of the set of vertices of $T$ is called an \emph{\ics{}} of $T$, if for any two elements $u$ and $v$ of $X$ the infimum of $u$ and $v$ is an element of $X$ as well. Figure~\ref{Fig:Example_ICS} shows an example of an \ics{} in a tree; note that we draw all trees with their root (least element) on top, in contrast to the usual convention for Hasse diagrams.

\begin{figure}[htbp]
 \centering
    \begin{tikzpicture}[scale=1]
	\draw (0,0)--(-2,-1)--(-3,-2)--(-3,-3);
	\draw (-2,-1)--(-1,-2)--(-1,-3);
	\draw (0,0)--(0,-1);
	\draw (0,0)--(2,-1)--(1,-2);
	\draw (2,-1)--(2,-2);
	\draw (2,-1)--(3,-2)--(2.5,-3);
	\draw (3,-2)--(3.5,-3);
	\node[fill=black,circle,inner sep=3pt] () at (0,0) {};
	\node[fill=black,circle,inner sep=3pt] () at (-2,-1) {};
	\node[fill=black,circle,inner sep=1pt] () at (0,-1) {};
	\node[fill=black,circle,inner sep=3pt] () at (2,-1) {}; 
	\node[fill=black,circle,inner sep=1pt] () at (-3,-2) {};
	\node[fill=black,circle,inner sep=3pt] () at (-1,-2) {};
	\node[fill=black,circle,inner sep=3pt] () at (1,-2) {};
	\node[fill=black,circle,inner sep=1pt] () at (2,-2) {};
	\node[fill=black,circle,inner sep=3pt] () at (3,-2) {};
	\node[fill=black,circle,inner sep=3pt] () at (-3,-3) {};
	\node[fill=black,circle,inner sep=1pt] () at (-1,-3) {};
	\node[fill=black,circle,inner sep=1pt] () at (2.5,-3) {};
	\node[fill=black,circle,inner sep=1pt] () at (3.5,-3) {};

     \end{tikzpicture} 
 \caption{A rooted tree and one of its \icss{}, indicated by larger circles.}
 \label{Fig:Example_ICS}
\end{figure}
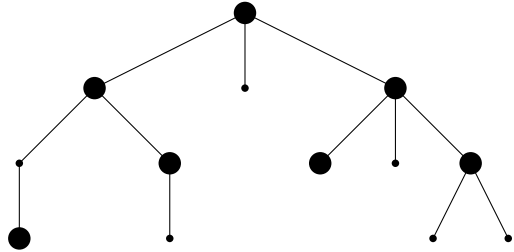

Let $\I(T)$ denote the number of  nonempty \icss{} in $T$. 
In \cite{klazar_1997_twelve} (see also \cite{klazar_1997_addendum}), Klazar studied the enumeration of various types of sets in plane trees, among them the number of nonempty infima closed sets. Specifically, he obtained an asymptotic formula for the sum of nonempty \icss{} over all plane trees with $n$ vertices.

It is also stated in Klazar's paper that the rooted tree of order $n$ with the greatest number of \icss{} is the path $P_n$, rooted at one of its ends. Indeed, it is easy to see that every 
 subset of $V(P_n)$ is infima closed, hence
$$
\max_{|T| = n} \I(T)=\I(P_n)=2^n-1.
$$
 It is interesting to note that the star $S_n$, rooted at its centre, which is in some sense the tree that differs most from a path, also has a fairly large number of \icss{}: indeed,
$$\I(S_n) = 2^{n-1} + n-1.$$
The problem of finding the trees with a given number of vertices for which $\I(T)$ attains its minimum was stated as an open question in \cite{klazar_1997_twelve}. Problems of this flavour, where the maximum or minimum number of a certain type of sets in trees is to be determined, have received considerable attention in the literature. See for instance \cite[Section 3]{szekely_2016_problems} for a recent survey. Some examples include independent sets and matchings (ordinary, maximal or maximum) \cite{wagner_2010_maxima,heuberger_2011_number,gorska_2007_trees,wilf_1986_number,zito_1991_structure}, subtrees \cite{szekely_2005_subtrees}, and different types of dominating sets \cite{krzywkowski_2018_graphs,brod_2008_recurrence,brod_2006_trees,krzywkowski_2013_trees,rote_2019_minimal,rote_2019_maximum}. In a very recent paper, Rosenfeld \cite{rosenfeld_2020_growth}, building on work of Rote \cite{rote_2019_maximum,rote_2019_minimal}, describes a very general approach to compute growth rates for problems of this kind.

The aim of this paper is to resolve the open question of Klazar. The structure of the trees that yield the minimum is quite interesting. They are similar to complete binary trees, except for the very last levels. This property follows mainly from Lemmas \ref{Lem:g8} and \ref{Lem:Form2}. A precise description will be our main result (Theorem~\ref{thm:main}). Figure~\ref{Fig:Example_Minimum} shows the tree with the smallest number of \icss{} among rooted trees with $15$ vertices as an example.

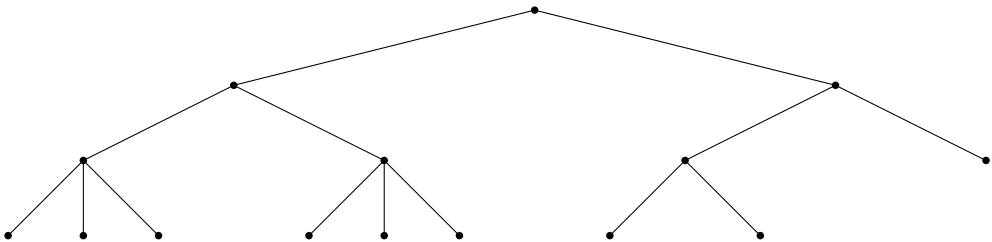
\begin{figure}[htbp]
 \centering
    \begin{tikzpicture}[scale=1]
	\draw (0,0)--(-4,-1);
	\draw (0,0)--(4,-1);
	\draw (-4,-1)--(-6,-2);
	\draw (-4,-1)--(-2,-2);
	\draw (4,-1)--(6,-2);
	\draw (4,-1)--(2,-2);
	\draw (-6,-2)--(-7,-3);
	\draw (-6,-2)--(-6,-3);
	\draw (-6,-2)--(-5,-3);
	\draw (-2,-2)--(-3,-3);
	\draw (-2,-2)--(-2,-3);
	\draw (-2,-2)--(-1,-3);
	\draw (2,-2)--(3,-3);
	\draw (2,-2)--(1,-3);

	\node[fill=black,circle,inner sep=1pt] () at (0,0) {};
	\node[fill=black,circle,inner sep=1pt] () at (-4,-1) {};
	\node[fill=black,circle,inner sep=1pt] () at (4,-1) {};
	\node[fill=black,circle,inner sep=1pt] () at (-6,-2) {}; 
	\node[fill=black,circle,inner sep=1pt] () at (-2,-2) {};
	\node[fill=black,circle,inner sep=1pt] () at (2,-2) {};
	\node[fill=black,circle,inner sep=1pt] () at (6,-2) {};
	\node[fill=black,circle,inner sep=1pt] () at (-7,-3) {};
	\node[fill=black,circle,inner sep=1pt] () at (-6,-3) {};
	\node[fill=black,circle,inner sep=1pt] () at (-5,-3) {};
	\node[fill=black,circle,inner sep=1pt] () at (-3,-3) {};
	\node[fill=black,circle,inner sep=1pt] () at (-2,-3) {};
	\node[fill=black,circle,inner sep=1pt] () at (-1,-3) {};
	\node[fill=black,circle,inner sep=1pt] () at (1,-3) {};
	\node[fill=black,circle,inner sep=1pt] () at (3,-3) {};
     \end{tikzpicture} 
 \caption{The rooted tree of order $15$ with the smallest number of \icss{}.}
 \label{Fig:Example_Minimum}
\end{figure}

Interestingly, there is not always a unique tree for which the minimum is attained: there are two such trees if the number of vertices is either $6$ or of the form $5m+2$ for a positive integer $m$. We will also provide an asymptotic analysis of the minimum value: specifically, we will show that there exists a constant $\alpha$ such that
$$
m_n = \min_{|T| = n} \I(T)=\Theta(\alpha^n).
$$
The numerical value of this constant is $\alpha \approx 1.66928\,37234\,96921\,49740\,26178$.

Preliminary results on the structure of the trees that attain the minimum are gathered in the following sections, leading up to the complete characterisation that is provided in Section~\ref{Sec:Main}. We conclude with the asymptotic analysis in our final section.

\section{Preliminaries}
\label{Sec:Pre}

\subsection{Notation}

Throughout this paper, we write $\I(T)$ for the number of nonempty \icss{} of a rooted tree $T$. It will also be necessary to consider several related quantities: specifically, $\I_0(T)$ denotes the number of nonempty \icss{} that do not contain the root, while $\I_1(T)$ is the number of \icss{} that contain the root. So we trivially have $\I(T) = \I_0(T) + \I_1(T)$. More generally, given arbitrary vertices $v,w$ of $T$, we write $\I_v(T)$ for the number of \icss{} that contain $v$ and $\I_{vw}(T)$ for the number of \icss{} that contain both $v$ and $w$.
The quantity of interest to us is the minimum
$$m_n = \min_{|T| = n} \I(T),$$
and trees that attain this minimum will also be called \emph{minimal trees}.

By a \emph{branch} of a rooted tree $T$, we mean any subtree of $T$ consisting of a vertex $v$ (the root of the branch) and all its descendants; in particular, the \emph{root branches} are the branches associated with the root's children. If $B_1,B_2,\dots,B_k$ are the root branches of $T$, then we write $T=[B_1,B_2,\dots,B_k]$. See Figure~\ref{Fig:Decomp} for an iterated use of this notation, which will prove useful in describing the structure of minimal trees.

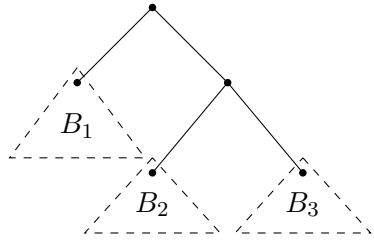
\begin{figure}[htbp]
 \centering
    \begin{tikzpicture}[scale=1]
        \draw (0,0)--(-1,1)--(-2,0);
	\draw (-1,-1.2)--(0,0)--(1,-1.2);
	\draw[dashed] (-2,0.2)--(-2.9,-1)--(-1.1,-1)--(-2,0.2);
	\draw[dashed] (-1,-1)--(-1.9,-2)--(-0.1,-2)--(-1,-1);
	\draw[dashed] (1,-1)--(1.9,-2)--(0.1,-2)--(1,-1);
	\node[fill=black,circle,inner sep=1pt] () at (-2,0) {};
	\node[fill=black,circle,inner sep=1pt] () at (-1,1) {};
        \node[fill=black,circle,inner sep=1pt] () at (0,0) {}; 
        \node[fill=black,circle,inner sep=1pt] () at (-1,-1.2) {};
        \node[fill=black,circle,inner sep=1pt] () at (1,-1.2) {};
        \node[fill=white,label=below:{$B_1$},circle,inner sep=0pt] () at (-2,-0.25) {};
        \node[fill=white,label=below:{$B_2$},circle,inner sep=0pt] () at (-1,-1.3) {};
        \node[fill=white,label=below:{$B_3$},circle,inner sep=0pt] () at (1,-1.3) {};
     \end{tikzpicture} 
 \caption{Decomposition of trees into branches: the figure shows the tree $[B_1,[B_2,B_3]]$.}
 \label{Fig:Decomp}
\end{figure}

Another useful decomposition involves two disjoint branches of a tree: given a rooted tree $S$ with two leaves $v$ and $w$ and two further rooted trees $A$ and $B$, we write $[A:_v\hspace{-0.14cm}S_w\hspace{-0.14cm}:B]$ for the tree obtained by merging the root of $A$ with $v$ and the root of $B$ with $w$, see Figure~\ref{Fig:Decomp_3}.

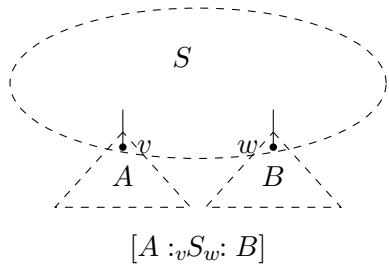
\begin{figure}[htbp]
 \centering
    \begin{tikzpicture}[scale=1]
        \draw[dashed] (-1+6,-1)--(-1.9+6,-2)--(-0.1+6,-2)--(-1+6,-1);
	\draw[dashed] (1+6,-1)--(1.9+6,-2)--(0.1+6,-2)--(1+6,-1);
	\draw[dashed] (6,-0.35) ellipse (2.5cm and 1cm);
	\draw (-1+6,-1.2)--(-1+6,-0.7);
	\draw (1+6,-1.2)--(1+6,-0.7);
	\node[fill=white,label=right:{$S$},circle,inner sep=0pt] () at (5.5,0) {};
        \node[fill=black,label=right:{$v$},circle,inner sep=1pt] () at (-1+6,-1.2) {};
        \node[fill=black,label=left:{$w$},circle,inner sep=1pt] () at (1+6,-1.2) {};
        \node[fill=white,label=below:{$A$},circle,inner sep=0pt] () at (-1+6,-1.3) {};
        \node[fill=white,label=below:{$B$},circle,inner sep=0pt] () at (1+6,-1.3) {};
        \node[fill=white,label=below:{$[A:_v\hspace{-0.14cm}S_w\hspace{-0.14cm}:B]$},circle,inner sep=0pt] () at (0+6,-2.2) {};
     \end{tikzpicture} 
 \caption{Illustration of the tree $[A:_v\hspace{-0.14cm}S_w\hspace{-0.14cm}:B]$.}
 \label{Fig:Decomp_3}
\end{figure}

Finally, we use the common notations $\deg(v)$ and $\dd(v,w)$ for the degree of a vertex $v$ and the distance between two vertices $v,w$ respectively.

\subsection{A recursion}

In order to determine the number of \icss{} of a rooted tree recursively from its root branches, we distinguish those \icss{} that do not contain the root from those that do. The former have to lie entirely in one of the branches, which immediately gives us
$$\I_0([B_1,B_2,\dots,B_k])=\sum_{j=1}^k\I(B_j).$$
On the other hand, every \ics{} of $T$ that contains the root induces a possibly empty \ics{} on each of the branches. Conversely, if we select a vertex subset in each of the branches that is infima closed, take the union of all those subsets and add the root, we obtain an \ics{} in $T$. 
This means that
$$
\I_1([B_1,B_2,\dots,B_k])=\prod_{j=1}^k\big(1+\I(B_j)\big).
$$
Consequently, we have the recursion
\begin{equation}\label{eq:recursion}
\I([B_1,B_2,\dots,B_k])=\sum_{j=1}^k\I(B_j)+\prod_{j=1}^k\big(1+\I(B_j)\big).
\end{equation}
This recursion will be used frequently throughout the paper. In particular, for $k=2$ we have
$$
\I([B_1,B_2])=\I(B_1)\I(B_2)+2\I(B_1)+2\I(B_2)+1=(\I(B_1)+2)(\I(B_2)+2) - 3.
$$
This also motivates the substitution $\J(T) = \I(T) + 2$, which will be useful later. In terms of the new invariant $\J$, the recursion simply becomes
\begin{equation}\label{eq:Jrec}
\J([B_1,B_2])=\J(B_1)\J(B_2)-1.
\end{equation}
Clearly, minimising $\I(T)$ is equivalent to minimising $\J(T)$, so we will sometimes work with $\J$ when it is more convenient.

The expression on the right side of~\eqref{eq:recursion} is clearly increasing in each of the $\I(B_j)$. This immediately yields the following simple yet useful lemma:

\begin{lem}\label{lem:branches}
Every branch of a minimal tree has to be a minimal tree as well.
\end{lem}

\begin{proof}
For root branches, this follows directly from~\eqref{eq:recursion}; for all other branches, we obtain it by induction.
\end{proof}

Lemma~\ref{lem:branches} allows us to determine minimal trees of small order efficiently by means of a computer program\footnote{Mathematica files are available on \url{https://arxiv.org/src/2008.10225v2/anc}.}. As it turns out, there is a unique minimal tree in most cases (which will be proven later). The only non-uniqueness appears for orders $6$ and $7$ and subsequently all minimal trees that contain a branch of order $7$ (branches of order $6$ cannot occur in a minimal tree of order greater than $6$, see Lemma~\ref{Lem:Forb_Br}). Let us write $M_n$ for the unique minimal tree of order $n$ if there is only one, and $M_n^1$ and $M_n^2$ if there are two minimal trees. This notation will be justified later by Theorem~\ref{thm:main}. For simplicity, we will use $\bullet$ to denote the tree of order $1$. Of course, $M_1 = \bullet$. Table~\ref{Tab:Min_Small} shows minimal trees of small order, expressed in terms of their root branches. The phenomenon that the root has degree $2$ in almost all cases will be proven later, see Lemma~\ref{Lem:g8}. It can also be observed from the table that minimal trees of order $4,9,19,39,\ldots$ play an important role. This will also be made precise at a later stage.

\begin{table}[htbp]
\centering
\begin{tabular}{|l|l||l|l||l|l||l|l|}
\hline
  $n$ & minimal tree(s)
& $n$ & minimal tree(s)
& $n$ & minimal tree(s) 
& $n$ & minimal tree(s)\\
\hline
\hline
$1$ & $\bullet$ 
& $2$ & $[\bullet]$ 
& $3$ & $[\bullet,\bullet]$
& $4$ & $[\bullet,\bullet,\bullet]$\\
$5$ & $[\bullet,\bullet,\bullet,\bullet]$
& $6$ & $[M_4,\bullet]$
&$7$ &$[M_4,\bullet,\bullet]$ 
& $8$ & $[M_4,M_3]$\\
&&& and $[M_3,\bullet,\bullet]$ && and $[M_3,M_3]$ &&\\
$9$ & $[M_4,M_4]$
&  $10$ & $[M_5,M_4]$
& $11$ & $[M_5,M_5]$
& $12$ & $[M_7^i,M_4]$\\
$13$ & $[M_8,M_4]$
& $14$ & $[M_9,M_4]$
& $15$ & $[M_9,M_5]$
& $16$ & $[M_{10},M_5]$\\
$17$ & $[M_{9},M_7^i]$
& $18$ & $[M_{9},M_8]$
& $19$ & $[M_{9},M_9]$
& $20$ & $[M_{10},M_9]$\\
$21$ & $[M_{11},M_9]$
& $22$ & $[M_{12}^i,M_9]$
& $23$ & $[M_{13},M_9]$
& $24$ & $[M_{14},M_9]$\\
$25$ & $[M_{15},M_9]$
& $26$ & $[M_{15},M_{10}]$
& $27$ & $[M_{17}^i,M_9]$
& $28$ & $[M_{18},M_9]$\\
$29$ & $[M_{19},M_9]$
& $30$ & $[M_{19},M_{10}]$
& $31$ & $[M_{19},M_{11}]$
& $32$ & $[M_{19},M_{12}^i]$\\
$33$ & $[M_{19},M_{13}]$
& $34$ & $[M_{19},M_{14}]$
& $35$ & $[M_{19},M_{15}]$
& $36$ & $[M_{20},M_{15}]$\\
$37$ & $[M_{19},M_{17}^i]$
& $38$ & $[M_{19},M_{18}]$
& $39$ & $[M_{19},M_{19}]$
& $40$ & $[M_{20},M_{19}]$\\
$41$ & $[M_{21},M_{19}]$
& $42$ & $[M_{22}^i,M_{19}]$
& $43$ & $[M_{23},M_{19}]$
& $44$ & $[M_{24},M_{19}]$\\
$45$ & $[M_{25},M_{19}]$
& $46$ & $[M_{26},M_{19}]$
& $47$ & $[M_{27}^i,M_{19}]$ 
& $48$ & $[M_{28},M_{19}]$\\ 
$49$ & $[M_{29},M_{19}]$
& $50$ & $[M_{30},M_{19}]$
& $51$ & $[M_{31},M_{19}]$ 
& $52$ & $[M_{32}^i,M_{19}]$\\
$53$ & $[M_{33},M_{19}]$
& $54$ & $[M_{34},M_{19}]$
& $55$ & $[M_{35},M_{19}]$
& $56$ & $[M_{35},M_{20}]$\\
$57$ & $[M_{37}^i,M_{19}]$
& $58$ & $[M_{38},M_{19}]$
& $59$ & $[M_{39},M_{19}]$
& $60$ & $[M_{39},M_{20}]$\\
\hline
\end{tabular}

\medskip

\caption{Minimal trees of small order.}
\label{Tab:Min_Small}
\end{table}

\begin{table}[htbp]
\begin{tabular}{|l||c|c|c|c|c|c|c|c|c|c|}
\hline
$n$ & 1 & 2 & 3 & 4 & 5 & 6 & 7 & 8 & 9 & 10 \\
\hline
$m_n$ & 1 & 3 & 6 & 11 & 20 & 36 & 61 & 101 & 166 & 283 \\
\hline
\hline
$n$ & 11 & 12 & 13 & 14 & 15 & 16 & 17 & 18 & 19 & 20 \\
\hline
$m_n$ & 481 & 816 & 1336 & 2181 & 3693 & 6267 & 10581 & 17301 & 28221 & 47877 \\
\hline
\end{tabular}

\medskip

\caption{Minimum values of $\I(T)$ for $|T| \leq 20$.}
\label{Tab:Min_Values}
\end{table}

\newpage

The following simple observation will often be useful:

\begin{lem}\label{lem:increasing_m}
The sequence $m_n = \min_{|T| = n} \I(T)$ is increasing in $n$. Thus, if $T_1$ and $T_2$ are both minimal trees for their respective orders, then $\I(T_1) < \I(T_2)$ if and only if $|T_1| < |T_2|$.
\end{lem}

\begin{proof}
Let $T$ be any minimal tree of order $n$. Removing a leaf from $T$, we obtain a tree $T'$ of order $n-1$ with fewer \icss{} than $T$. Thus
$$m_{n-1} \leq \I(T') < \I(T) = m_n.$$
\end{proof}

\subsection{General branches}

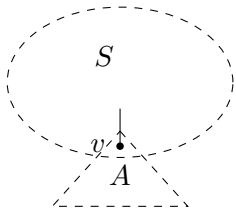
\begin{figure}[htbp]
 \centering

\begin{tikzpicture}[scale=1]
	\draw[dashed] (1+6,-1)--(1.9+6,-2)--(0.1+6,-2)--(1+6,-1);
	\draw[dashed] (7,-0.35) ellipse (1.5cm and 1cm);
	\draw (1+6,-1.2)--(1+6,-0.7);
	\node[fill=white,label=right:{$S$},circle,inner sep=0pt] () at (6.5,0) {};
        \node[fill=black,label=left:{$v$},circle,inner sep=1pt] () at (1+6,-1.2) {};
        \node[fill=white,label=below:{$A$},circle,inner sep=0pt] () at (1+6,-1.3) {};
     \end{tikzpicture}
 \caption{Lemma~\ref{Lem:S_A}, part (i).}
 \label{Fig:Decomp_2}
\end{figure}

Our next lemma provides relations involving arbitrary branches.
\begin{lem}
\label{Lem:S_A}
\
\begin{itemize}
 \item[(i)] Suppose that a tree $T$ can be decomposed as in Figure~\ref{Fig:Decomp_2},
where $v$ is a leaf in $S$ and $A$ is the branch rooted at $v$. Then we have $\I(T)=\I_v(S)\I(A)+\I(S-v)$. 
\item[(ii)] Consider the tree $T = [A:_v\hspace{-0.14cm}S_w\hspace{-0.14cm}:B]$ as in Figure~\ref{Fig:Decomp_3}, where $A,B$ and $S$ are rooted trees, and 
$v$ and $w$ are leaves of $S$. The number of nonempty \icss{} of this composite tree can be expressed as
$$
\I(T) =\I_{vw}(S)\I(A)\I(B)+(\I_{v}(S)-\I_{vw}(S))\I(A)+(\I_{w}(S)-\I_{vw}(S))\I(B)+\I(S-\{v,w\}).
$$
\end{itemize}
\end{lem}
\begin{proof}\ 
\begin{itemize}
\item[(i)] We can partition the set of nonempty \icss{} of $T$ into those that contain a vertex of $A$ and those that do not. The latter is clearly counted by $\I(S-v)$. If an \ics{} of $T$ contains vertices of $A$, then those have to form an \ics{} in $A$ as well. Moreover, as the infimum of any vertex $w$ of $S-v$ and any vertex of $A$ is the same as the infimum of $w$ and $v$, the set of vertices outside of $A$ have to form an \ics{} of $S$ together with $v$. The converse is also true: take any \ics{} of $S$ that contains $v$ and replace $v$ by an arbitrary nonempty \ics{} of $A$ to obtain an \ics{} of $T$. Therefore, the number of \icss{} of $T$ that contain one or more vertices of $A$ is $\I_v(S)\I(A)$, which completes the proof of the formula.
\item[(ii)] This statement follows in a straightforward fashion by applying the argument of (i) twice and dividing the set of nonempty \icss{} of $T$ into four subsets depending on whether or not they contain vertices of $A$ and $B$ respectively.
\end{itemize}
\end{proof}

\subsection{Some useful inequalities}

The following technical lemma, which provides three different inequalities, will be required later to estimate the effect of certain transformations.

\begin{lem}
\label{Lem:I1_I0}
Let $T$ be a rooted tree.
\begin{itemize}
\item[(i)]We have
$$\I_1(T) > \I_0(T).$$
\item[(ii)]If $v$ is a leaf of $T$, then
$$
3\I_v(T)\geq \I(T)+2.
$$
\item[(iii)] If $v$ and $w$  are distinct leaves of $T$, then 
$$
3\I_{vw}(T)\geq \I_v(T).
$$
\end{itemize} 
\end{lem}
\begin{proof}\ 
\begin{itemize}
\item[(i)] Every \ics{} that does not contain the root can be turned into an \ics{} containing the root by simply adding the root. Moreover, the set containing only the root is always an \ics{}. Therefore, $\I_1(B) \geq \I_0(B) + 1 > \I_0(B)$.
\item[(ii)]
We use induction with respect to the order for this part. For $|T|=1$, both sides of the inequality are equal to $3$. Now assume that the statement holds for trees of order at most $k$, and consider a tree $T$ with $k+1$ vertices. 
We can decompose $T$ as in Figure~\ref{fig:lemI1I0_ii}. Here, $A$ is the root branch that contains $v$, and $B$ constitutes the rest of the tree (possibly only the root).

\begin{figure}[htbp]
\centering
\begin{tikzpicture}[scale=1]
	\draw[dashed] (-1,-0.8)--(-1.9,-2)--(-0.1,-2)--(-1,-0.8);
	\draw[dashed] (-0.1,0.2)--(1.9,-2)--(0.2,-2)--(-0.1,0.2);
	\draw (-1,-1)--(0,0);
	\node[fill=black,circle,inner sep=1pt] () at (0,0) {};
	\node[fill=black,circle,inner sep=1pt] () at (-1,-1) {};
	\node[fill=white,label=below:{$A$},circle,inner sep=0pt] () at (-1,-1.2) {};
        \node[fill=white,label=below:{$B$},circle,inner sep=0pt] () at (0.6,-1.2) {};
        \node[fill=black,circle,label=below:{$v$},inner sep=1pt] () at (-1.7,-1.9) {};
     \end{tikzpicture}
\caption{Decomposition of $T$ in the proof of Lemma~\ref{Lem:I1_I0}, part (ii).}\label{fig:lemI1I0_ii}
\end{figure}
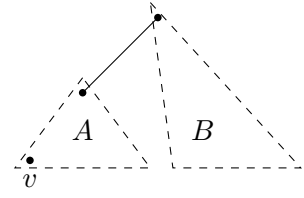

Using the induction hypothesis and part (i) (which also implies that $\I(B) = \I_1(B) + \I_0(B) < 2\I_1(B)$), we obtain
\begin{align*}
\I(T)+2 &= \I_0(T) + \I_1(T) + 2 \\
&= \I(A) + \I_0(B) + \I_1(B)(1+\I(A)) + 2\\
&=\I(A)+\I(B)+\I_1(B)\I(A)+2\\
&\leq\I(A)+2\I_1(B)+\I_1(B)\I(A)+2 \\
&=(\I(A)+2)(\I_1(B)+1)\\
&\leq 3\I_v(A)(\I_1(B)+1)=3\I_v(T).
\end{align*}
The last equality follows by splitting the set of \icss{} containing $v$ into those that also contain the root of $T$ and those that do not contain it. There are $\I_v(A)\I_1(B)$ sets of the former type, and $\I_v(A)$ of the latter. The same idea will be applied repeatedly in the following.
\item[(iii)] Again, we use induction on the order of $T$. For $|T|=3$, we have only one possible tree with two leaves $v$ and $w$, and we obtain
$$
3\I_{vw}(T)=3=\I_v(T).
$$
Now let us assume that the inequality is satisfied whenever $3\leq |T|\leq k$, and consider a tree $T$ such that $|T|=k+1$. If $v$ and $w$ are in the same root branch $A$, then $3\I_{vw}(A)\geq \I_v(A)$ by the induction hypothesis and thus 
$$
3\I_{vw}(T)
=3\I_{vw}(A)(\I_1(T-A)+1) \geq \I_v(A)(\I_1(T-A)+1)= \I_v(T).
$$
Otherwise $v$ and $w$ belong to different root branches $A$ and $B$, respectively. In this case we have
\begin{align*}
3\I_{vw}(T)-\I_v(T)
&=3\I_v(A)\I_w(B)\I_1(T-A-B)-\I_v(A)(\I_1(T-A)+1)\\
&=3\I_v(A)\I_w(B)\I_1(T-A-B)-\I_v(A)((\I(B)+1)\I_1(T-A-B)+1)\\
&=\I_v(A)((3\I_w(B)-\I(B)-1)\I_1(T-A-B)-1)\\
&\geq\I_v(A)(3\I_w(B)-\I(B)-2)\geq 0,
\end{align*}
since we can apply (ii) to $w$ and $B$ to obtain $3\I_w(B)-\I(B)-2\geq 0$.
\end{itemize}
\end{proof}

\section{The structure of minimal trees}

In this section, we establish structural properties of minimal trees, building towards their complete characterisation that will be provided in the following section. Specifically, we first determine restrictions on the possible degrees of vertices in a minimal tree, then on the sizes and heights of branches. Finally, we will be able to exclude a list of trees as potential branches.

\subsection{Degrees}

Our first lemma on the structure of minimal trees provides information on the  vertex degrees.
\begin{lem}
\label{Lem:g8}
Consider a minimal tree $T$ with $n$ vertices.
\begin{itemize}
\item[(i)] If  $v$ is a vertex of $T$ other than the root, then $\deg(v)\neq 2$. In other words, $v$ is a leaf or it has at least two children.
 \item[(ii)] If $n\geq 8$, then the root degree $\deg(\rr(T))$ is $2$.
 \end{itemize}
\end{lem}
\begin{proof}\ 
\begin{itemize}
\item[(i)] Suppose that $T$ has a non-root vertex $v$ of degree $2$, and let $H$ be the branch rooted at $v$'s parent vertex. We prove that $H$ cannot be minimal by transforming it to a new tree $H'$ as shown in Figure~\ref{fig:h_transform}: $v$'s unique child is merged with $v$, and a new child is added to $v$'s parent vertex.
\begin{figure}[htbp]
\centering
\begin{tikzpicture}[scale=0.75]
	\draw[dashed] (0,0.2)--(-1,-3)--(0.4,-3)--(0,0.2);
	\draw[dashed] (1.5,-1.8)--(0.5,-3)--(2.5,-3)--(1.5,-1.8);
	\draw (0,0)--(1,-1)--(1.5,-2);
	\node[fill=black,circle,inner sep=1pt] () at (0,0) {};
	\node[fill=black,label=above:{$v$},circle,inner sep=1pt] () at (1,-1) {};
	\node[fill=black,circle,inner sep=1pt] () at (1.5,-2) {};
	\node[fill=white,label=below:{$A$},circle,inner sep=0pt] () at (-0.25,-1.2) {};
        \node[fill=white,label=below:{$B$},circle,inner sep=0pt] () at (1.5,-2.2) {};
        \node[fill=white,label=below:{$H$},circle,inner sep=0pt] () at (0.5,-3.2) {};
        \draw[->, very thick](3,-1)--(4,-1);
	\draw[dashed] (0+6,0.2)--(-1+6,-3)--(0.4+6,-3)--(0+6,0.2);
	\draw[dashed] (1.5+6,-1.8+1)--(0.5+6,-3+1)--(2.5+6,-3+1)--(1.5+6,-1.8+1);
	\draw (0+6,0)--(1.5+6,-1);
	\draw (0+6,0)--(1+8.5,-1);
	\node[fill=black,circle,inner sep=1pt] () at (0+6,0) {};
	\node[fill=black,label=above:{$v$},circle,inner sep=1pt] () at (1.5+6,-1) {};
	\node[fill=black,circle,inner sep=1pt] () at (1+8.5,-1) {};
	\node[fill=white,label=below:{$A$},circle,inner sep=0pt] () at (-0.25+6,-1.2) {};
        \node[fill=white,label=below:{$B$},circle,inner sep=0pt] () at (1.5+6,-1.2) {};
        \node[fill=white,label=below:{$H'$},circle,inner sep=0pt] () at (0.5+6,-3.2) {};
     \end{tikzpicture}
\caption{Transformation of $H$.}\label{fig:h_transform}
\end{figure}
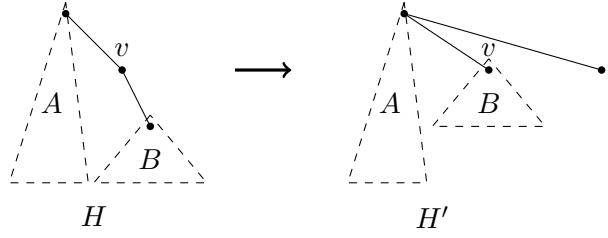
We find that
\begin{align*}
\I(H)-\I(H')
&= \I_0(A)+2\I(B)+1+\I_1(A)(2\I(B)+2)\\
&\quad-\I_0(A)-\I(B)-1-2\I_1(A)(\I(B)+1)=\I(B)>0,
\end{align*}
which shows that $H$ is not minimal, contradicting Lemma~\ref{lem:branches}.

\item[(ii)]

Assume now that $n\geq 8$, 
and that $k = \deg(\rr(T))\neq 2$. We prove that there is a rooted tree $T'$ such that $|T'|=|T|=n$ and 
$\I(T')<\I(T)$, a contradiction from which the desired statement follows immediately. We have the following cases:

\begin{itemize}
\item[(a)] Suppose that the root degree $k$ is $1$, i.e., $T$ has only one branch $A$. 
We decompose this branch further, as shown in Figure~\ref{fig:T_further_decomp}, into the rightmost branch $C$ and the rest $B$ (note that its root has at least two children by part (i)).
Taking $T'=[B,C]$, we have
\begin{align*}
\I(T)-\I(T') &=2\big(\I_0(B)+\I(C)+\I_1(B)(\I(C)+1)\big)+1-\I(B)\I(C)-2\I(B)-2\I(C)-1\\
&=(2\I_1(B)-\I(B))\I(C) >0,
\end{align*}
where the final inequality is a consequence of Lemma~\ref{Lem:I1_I0}, part (i). This is what we wanted to obtain.

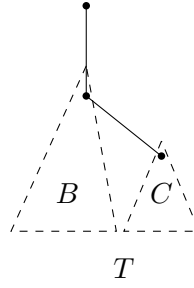
\begin{figure}[htbp]
\centering
 \begin{tikzpicture}[scale=1]
	\draw[dashed] (0,0.2)--(-1,-2)--(0.4,-2)--(0,0.2);
	\draw[dashed] (1,-0.8)--(0.5,-2)--(1.5,-2)--(1,-0.8);
	\draw (0,-0.2)--(1,-1);
	\draw (0,-0.2)--(0,1);
	\node[fill=black,circle,inner sep=1pt] () at (0,-0.2) {};
	\node[fill=black,circle,inner sep=1pt] () at (1,-1) {};
	\node[fill=black,circle,inner sep=1pt] () at (0,1) {};
	\node[fill=white,label=below:{$B$},circle,inner sep=0pt] () at (-0.25,-1.2) {};
        \node[fill=white,label=below:{$C$},circle,inner sep=0pt] () at (1,-1.2) {};
        \node[fill=white,label=below:{$T$},circle,inner sep=0pt] () at (0.5,-2.2) {};
      \end{tikzpicture}
\caption{Decomposition of $T$ in case (a).}\label{fig:T_further_decomp}
\end{figure}

\item[(b)] Now assume that $k\geq 3$. Here and in the following cases, we let $T_1,T_2,\ldots,T_k$ be the root branches of $T$, so that $T = [T_1,T_2,\ldots,T_k]$. Without loss of generality, we can assume that $|T_1| \geq |T_2| \geq \cdots \geq |T_k|$. Each branch is a minimal tree, so Lemma~\ref{lem:increasing_m} shows that $\I(T_1) \geq \I(T_2) \geq \cdots \geq \I(T_k)$. Let us first consider the case that $|T_2|\geq 2$. Define $L=[T_1,T_4,T_5,\dots,T_k]$, and consider the tree $T'$ constructed as in Figure~\ref{fig:T'construction}.
\begin{figure}[htbp]
\centering
 \begin{tikzpicture}[scale=1]
	\draw[dashed] (0,0.2)--(-0.5,-2)--(0.4,-2)--(0,0.2);
	\draw[dashed] (1,-0.8)--(0.5,-2)--(1.5,-2)--(1,-0.8);
	\draw[dashed] (-1,1.2)--(-0.75,-2)--(-1.75,-2)--(-1,1.2);
	\draw (0,-0.1)--(1,-1);
	\draw (0,-0.1)--(-1,1);
	\node[fill=black,circle,inner sep=1pt] () at (0,-0.1) {};
	\node[fill=black,circle,inner sep=1pt] () at (1,-1) {};
	\node[fill=black,circle,inner sep=1pt] () at (-1,1) {};
	\node[fill=white,label=below:{$L$},circle,inner sep=0pt] () at (-1.25,-1.2) {};
	\node[fill=white,label=below:{$T_2$},circle,inner sep=0pt] () at (0,-1.2) {};
        \node[fill=white,label=below:{$T_3$},circle,inner sep=0pt] () at (1,-1.2) {};
        \node[fill=white,label=below:{$T'$},circle,inner sep=0pt] () at (0,-2.2) {};
      \end{tikzpicture}
\caption{Construction of $T'$ in case (b).}\label{fig:T'construction}
\end{figure}
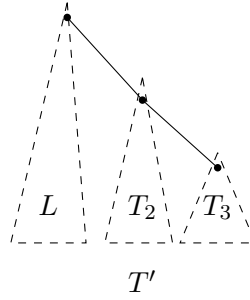

Repeated application of the recursions for $\I_1$ and $\I_0$ yields 
 \begin{align*}
\I(T)-\I(T')
 &=\I_0(L)+\I(T_2)+\I(T_3)+\I_1(L)(\I(T_2)+1)(\I(T_3)+1)-\I_0(L)-\I(T_2)\\
 &\qquad -\I(T_3)-\I_1(T_2)\I(T_3)-\I_1(L)\big(\I(T_2)+\I(T_3)+\I_1(T_2)\I(T_3)+1\big)\\
 &=\I_1(L)\I_0(T_2)\I(T_3)-\I_1(T_2)\I(T_3)\\
 &\geq(\I_1(L)-\I_1(T_2))\I(T_3) \\
 &\geq (\I(T_1)-\I_1(T_2))\I(T_3)>0,
 \end{align*}
which completes the proof in this case.
 
\item[(c)] We are left with the situation that $|T_2| = 1$, so $|T_2|=|T_3|=\dots=|T_k|=1$, i.e. $T = [T_1,\bullet,\bullet,\ldots,\bullet]$. Assume now that $k\geq 4$. We can once again consider a different tree $T'$, defined by
$T'=[T_1,\bullet,\ldots,\bullet,[\bullet,\bullet]],$
replacing three of the singleton branches of $T$ by a branch of order $3$. We have
 \begin{align*}
 \I(T)-\I(T')
 &= k-1+ \I(T_1) + 2^{k-1}(\I(T_1)+1)\\
 &\qquad- \big(6+(k-4)+\I(T_1)+7 \cdot 2^{k-4}(\I(T_1)+1)\big)\\
 &= 2^{k-4}(\I(T_1)+1)- 3>0.
 \end{align*}
The last inequality trivially holds if $k \geq 5$. If $k=4$, we simply need to note that $|T_1| \geq 2$ and thus $\I(T_1) \geq 3$, since we are assuming that $T$ has at least eight vertices. Note that this construction also shows that a star of order greater than $5$ cannot be a minimal tree.

\item[(d)] It only remains to consider the case $k=3$ and $|T_2|=|T_3|=1$. For this we have to consider subcases depending on the structure of $T_1$.

 If the branch $T_1$ is not a star (in other words, if there are other vertices than the root and its children), then decompose $T_1$ into a root branch $B$ with at least two vertices and the rest, which is denoted by $A$. We define $T'$ as in Figure~\ref{fig:T'construction_d}: it has two root branches, one of which is $B$, while the other is $A$ with two additional leaves attached to its root.
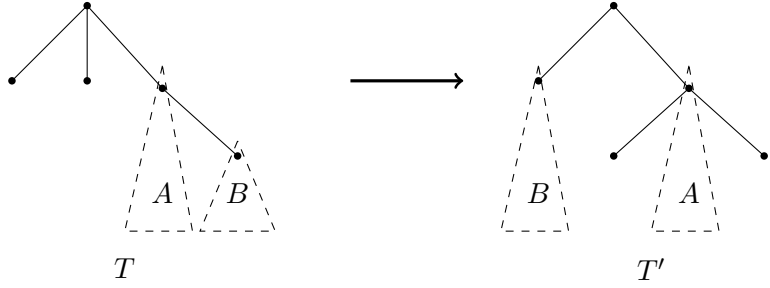
\begin{figure}[htbp]
\centering
 \begin{tikzpicture}[scale=1]
	\draw[dashed] (0,0.2)--(-0.5,-2)--(0.4,-2)--(0,0.2);
	\draw[dashed] (1,-0.8)--(0.5,-2)--(1.5,-2)--(1,-0.8);
	\draw (-1,0)--(-1,1)--(-2,0);
	\draw (0,-0.1)--(1,-1);
	\draw (0,-0.1)--(-1,1);
	\node[fill=black,circle,inner sep=1pt] () at (-1,0) {};
	\node[fill=black,circle,inner sep=1pt] () at (-2,0) {};
	\node[fill=black,circle,inner sep=1pt] () at (0,-0.1) {};
	\node[fill=black,circle,inner sep=1pt] () at (1,-1) {};
	\node[fill=black,circle,inner sep=1pt] () at (-1,1) {};
	\node[fill=white,label=below:{$A$},circle,inner sep=0pt] () at (0,-1.2) {};
    \node[fill=white,label=below:{$B$},circle,inner sep=0pt] () at (1,-1.2) {};
    \node[fill=white,label=below:{$T$},circle,inner sep=0pt] () at (-0.5,-2.2) {};
    	\draw[->,very thick] (2.5,0)--(4,0);
	\draw[dashed] (5+0,0.2)--(5-0.5,-2)--(5+0.4,-2)--(5+0,0.2);
	\draw[dashed] (7+0,0.2)--(7-0.5,-2)--(7+0.4,-2)--(7+0,0.2);
	\draw (7-1,1)--(7-2,0);
	\draw (7,-0.1)--(6,-1);
	\draw (7+0,-0.1)--(7+1,-1);
	\draw (7+0,-0.1)--(7-1,1);
	\node[fill=black,circle,inner sep=1pt] () at (6,-1) {};
	\node[fill=black,circle,inner sep=1pt] () at (7-2,0) {};
	\node[fill=black,circle,inner sep=1pt] () at (7+0,-0.1) {};
	\node[fill=black,circle,inner sep=1pt] () at (7+1,-1) {};
	\node[fill=black,circle,inner sep=1pt] () at (7-1,1) {};
	\node[fill=white,label=below:{$B$},circle,inner sep=0pt] () at (5+0,-1.2) {};
    \node[fill=white,label=below:{$A$},circle,inner sep=0pt] () at (6+1,-1.2) {};
    \node[fill=white,label=below:{$T'$},circle,inner sep=0pt] () at (6.5+0,-2.2) {};        
    \end{tikzpicture}
\caption{Construction of $T'$ in case (d).}\label{fig:T'construction_d} 
\end{figure}
Since $|B|\geq 2$ and thus $\I(B)\geq 3$, we obtain (again by repeated applications of the recursions for $\I_0$ and $\I_1$)
\begin{align*}
\I(T)- \I(T')
&= 2+\I(A)+\I(B)+\I_1(A)\I(B) + 4\big(\I(A)+\I(B)+\I_1(A)\I(B)+1\big)\\
&\qquad -\I(B)-2-\I(A)-3\I_1(A)-(\I(B)+1)(3+\I(A)+3\I_1(A))\\
&=1+\I(B)+(\I(B)-3)(\I_1(A)-\I_0(A))>0.
\end{align*}
If $T_1$ is a star, then its order is at least $5$ (since $T$ has at least eight vertices). As mentioned at the end of the previous case, $T_1$ cannot be minimal, contradicting Lemma~\ref{lem:branches}, unless its order is exactly $5$. However, if $T_1$ is a star of order $5$, then $T$ is not minimal either, as can be seen from Table~\ref{Tab:Min_Small}.
\end{itemize}
\end{itemize}
\end{proof}

The second part of Lemma~\ref{Lem:g8} implies (by Lemma~\ref{lem:branches})  that if $v$ is any vertex in a minimal tree such that the branch rooted at $v$ has order at least $8$, then $v$ has exactly two children.

\subsection{Branch sizes} Next, we determine several properties of branch sizes in a minimal tree.

\begin{lem}
\label{Lem:Form2}
Let $T$ be a minimal tree.
\begin{itemize}
\item[(i)] If $T$ can be decomposed as $T = [[A,B],C]$, as shown in Figure~\ref{fig:ABC_decomp}, then
\begin{equation}
\label{Eq:ABC}
\max\{|A|,|B|\}\leq|C|.
\end{equation}
\item[(ii)] If $T$ can be decomposed as $T = [[A,B],[[D,E],C]]$, as shown in Figure~\ref{fig:ABCDE_decomp}, then
\begin{equation}
\label{Eq:ABCD3}
\max\{|D|,|E|\}\leq\min\{|A|,|B|\}.
\end{equation}
\end{itemize}
\end{lem}
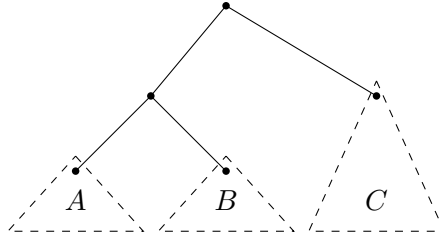
\begin{figure}[htbp]
\centering
\begin{tikzpicture}[scale=1]
	\draw (4-1,-1.2)--(4+0,0)--(5+1,-1.2);
	\draw (3-1,-2.2)--(3+0,-1.2)--(3+1,-2.2);
	\draw[dashed] (5+1,-1)--(5+1.9,-3)--(5+0.1,-3)--(5+1,-1);
	\draw[dashed] (3+1,-2)--(3+1.9,-3)--(3+0.1,-3)--(3+1,-2);
	\draw[dashed] (1+1,-2)--(1+1.9,-3)--(1+0.1,-3)--(1+1,-2);
    \node[fill=black,circle,inner sep=1pt] () at (1+1,-2.2) {}; 
    \node[fill=black,circle,inner sep=1pt] () at (3+1,-2.2) {}; 
    \node[fill=black,circle,inner sep=1pt] () at (4+0,0) {}; 
    \node[fill=black,circle,inner sep=1pt] () at (4-1,-1.2) {};
    \node[fill=black,circle,inner sep=1pt] () at (5+1,-1.2) {};
    \node[fill=white,label=below:{$C$},circle,inner sep=0pt] () at (5+1,-2.3) {};
    \node[fill=white,label=below:{$B$},circle,inner sep=0pt] () at (3+1,-2.3) {};
    \node[fill=white,label=below:{$A$},circle,inner sep=0pt] () at (1+1,-2.3) {};
    \end{tikzpicture} 
\caption{The decomposition $T = [[A,B],C]$ for Lemma \ref{Lem:Form2}.}\label{fig:ABC_decomp}
\end{figure}
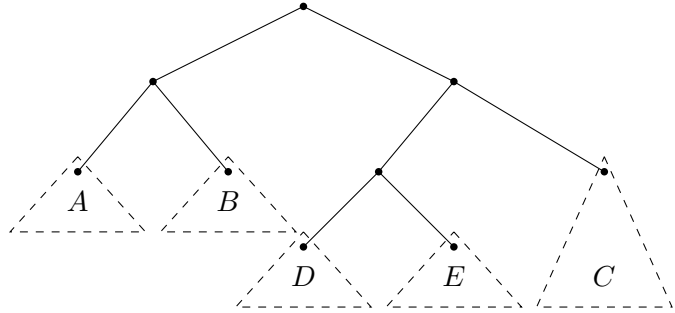
\begin{figure}[htbp]
\centering
\begin{tikzpicture}[scale=1]
    \draw (0,0)--(2,1);
	\draw (-1,-1.2)--(0,0)--(1,-1.2);
	\draw[dashed] (-1,-1)--(-1.9,-2)--(-0.1,-2)--(-1,-1);
	\draw[dashed] (1,-1)--(1.9,-2)--(0.1,-2)--(1,-1);
    \node[fill=black,circle,inner sep=1pt] () at (0,0) {}; 
    \node[fill=black,circle,inner sep=1pt] () at (-1,-1.2) {};
    \node[fill=black,circle,inner sep=1pt] () at (1,-1.2) {};
    \node[fill=black,circle,inner sep=1pt] () at (2,1) {};
    \node[fill=white,label=below:{$A$},circle,inner sep=0pt] () at (-1,-1.3) {};
    \node[fill=white,label=below:{$B$},circle,inner sep=0pt] () at (1,-1.3) {};
    \draw (4+0,0)--(2,1);
	\draw (4-1,-1.2)--(4+0,0)--(5+1,-1.2);
	\draw (3-1,-2.2)--(3+0,-1.2)--(3+1,-2.2);
	\draw[dashed] (5+1,-1)--(5+1.9,-3)--(5+0.1,-3)--(5+1,-1);
	\draw[dashed] (3+1,-2)--(3+1.9,-3)--(3+0.1,-3)--(3+1,-2);
	\draw[dashed] (1+1,-2)--(1+1.9,-3)--(1+0.1,-3)--(1+1,-2);
    \node[fill=black,circle,inner sep=1pt] () at (1+1,-2.2) {}; 
    \node[fill=black,circle,inner sep=1pt] () at (3+1,-2.2) {}; 
    \node[fill=black,circle,inner sep=1pt] () at (4+0,0) {}; 
    \node[fill=black,circle,inner sep=1pt] () at (4-1,-1.2) {};
    \node[fill=black,circle,inner sep=1pt] () at (5+1,-1.2) {};
    \node[fill=white,label=below:{$C$},circle,inner sep=0pt] () at (5+1,-2.3) {};
    \node[fill=white,label=below:{$E$},circle,inner sep=0pt] () at (3+1,-2.3) {};
    \node[fill=white,label=below:{$D$},circle,inner sep=0pt] () at (1+1,-2.3) {};
    \end{tikzpicture} 
\caption{The decomposition $T = [[A,B],[[D,E],C]]$ for Lemma \ref{Lem:Form2}.}\label{fig:ABCDE_decomp}
\end{figure}
\begin{proof}
\
\begin{itemize}
\item[(i)] If $T=[[A,B],C]$, then we have, by~\eqref{eq:Jrec},
$$\I(T) = \J(T) - 2 =  (\J(A)\J(B)-1)\J(C)-3=\J(A)\J(B)\J(C)-3 -\J(C).$$
Note that the product $\J(A)\J(B)\J(C)$ remains invariant under permutations of $A$, $B$ and $C$. By minimality of $T$, $\J(C) \geq \J(A)$ and $\J(C) \geq \J(B)$, as one could otherwise interchange $A$ and $C$ ($B$ and $C$, respectively) to obtain a tree of the same order with fewer \icss{}. Lemma~\ref{lem:increasing_m} now yields the statement.
\item[(ii)] If $T = [[A,B],[[D,E],C]]$, then we obtain, again by~\eqref{eq:Jrec},
\begin{align*}
\I(T)
&=(\J(A)\J(B)-1)\big((\J(D)\J(E)-1)\J(C)-1\big)-3\\
&=\J(A)\J(B)\J(C)\J(D)\J(E) + \J(C)-2 \\
&\qquad - \big( \J(C)(\J(A)\J(B) + \J(D)\J(E)) + \J(A)\J(B) \big).
\end{align*}
Note that $\J(A)\J(B)\J(C)\J(D)\J(E)+\J(C)-2$ remains invariant under permutations of $A$, $B$, $D$ and $E$. The remaining expression
$$\J(C)(\J(A)\J(B) + \J(D)\J(E)) + \J(A)\J(B)$$
must therefore be maximal over all such permutations. By the rearrangement inequality, $\J(A)\J(B) + \J(D)\J(E)$ can only be maximal if either $\J(A),\J(B) \geq \J(D),\J(E)$ or $\J(A),\J(B) \leq \J(D),\J(E)$. The product $\J(A)\J(B)$ is clearly maximal if and only if $\J(A),\J(B) \geq \J(D),\J(E)$. This argument combined with Lemma~\ref{lem:increasing_m} yields the desired conclusion.
\end{itemize}
\end{proof}
Lemma~\ref{Lem:Form2} is just a special case of the following more general lemma. It will serve as the initial case of an induction proof.
\begin{lem}
\label{Lem:GenComp}
Let $T$ be a minimal tree that can be decomposed as $[[A,B]:_v\hspace{-0.14cm}S_w\hspace{-0.14cm}:[C,D]]$, i.e., as shown in Figure~\ref{fig:LemGenComp}, where $v$ and $w$ are leaves of $S$.
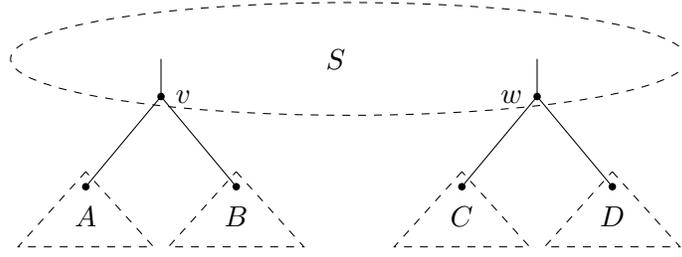
\begin{figure}[htbp]
\centering
\begin{tikzpicture}[scale=1]
    \draw (0,0)--(0,0.5);
	\draw (-1,-1.2)--(0,0)--(1,-1.2);
	\draw[dashed] (-1,-1)--(-1.9,-2)--(-0.1,-2)--(-1,-1);
	\draw[dashed] (1,-1)--(1.9,-2)--(0.1,-2)--(1,-1);
    \node[fill=black,label=right:{$v$},circle,inner sep=1pt] () at (0,0) {}; 
    \node[fill=black,circle,inner sep=1pt] () at (-1,-1.2) {};
    \node[fill=black,circle,inner sep=1pt] () at (1,-1.2) {};
    \node[fill=white,label=below:{$A$},circle,inner sep=0pt] () at (-1,-1.3) {};
    \node[fill=white,label=below:{$B$},circle,inner sep=0pt] () at (1,-1.3) {};
    \draw (5+0,0)--(5+0,0.5);
	\draw (5-1,-1.2)--(5+0,0)--(5+1,-1.2);
	\draw[dashed] (5-1,-1)--(5-1.9,-2)--(5-0.1,-2)--(5-1,-1);
	\draw[dashed] (5+1,-1)--(5+1.9,-2)--(5+0.1,-2)--(5+1,-1);
    \node[fill=black,label=left:{$w$},circle,inner sep=1pt] () at (5+0,0) {}; 
    \node[fill=white,label=right:{$S$},circle,inner sep=1pt] () at (2,0.5) {}; 
    \node[fill=black,circle,inner sep=1pt] () at (5-1,-1.2) {};
    \node[fill=black,circle,inner sep=1pt] () at (5+1,-1.2) {};
    \node[fill=white,label=below:{$C$},circle,inner sep=0pt] () at (5-1,-1.3) {};
    \node[fill=white,label=below:{$D$},circle,inner sep=0pt] () at (5+1,-1.3) {};
    \draw[dashed] (2.5,0.5) ellipse (4.5cm and 0.75cm);
    \end{tikzpicture} 
\caption{The setting of Lemma~\ref{Lem:GenComp}.}\label{fig:LemGenComp}
\end{figure}
\begin{itemize}
\item[(i)] We have 
$$
\max\{|A|,|B|\}\leq\min\{|C|,|D|\}
$$
or
$$
\max\{|C|,|D|\}\leq\min\{|A|,|B|\}.
$$
\item[(ii)] If furthermore $\dd(v,\rr(T))>\dd(w,\rr(T))$, then we must have
$$
\max\{|A|,|B|\}\leq\min\{|C|,|D|\}.
$$
\end{itemize}
\end{lem}
\begin{proof}\ 
\begin{itemize}
\item[(i)] We apply the second part of Lemma~\ref{Lem:S_A} and~\eqref{eq:Jrec} to get
\begin{align*}
\I(T)
&=\I_{vw}(S)(\J(A)\J(B)-3)(\J(C)\J(D)-3)+\I(S-\{v,w\})\\
&\qquad+(\I_v(S)-\I_{vw}(S))(\J(A)\J(B)-3)+(\I_w(S)-\I_{vw}(S))(\J(C)\J(D)-3)\\
&=\I_{vw}(S)(\J(A)\J(B)\J(C)\J(D)+15) - 3(\I_v(S)+\I_w(S))+\I(S-\{v,w\}) \\
&\qquad-(4\I_{vw}(S)-\I_v(S))\J(A)\J(B) - (4\I_{vw}(S)-\I_w(S))\J(C)\J(D).
\end{align*}
Note that $\I_{vw}(S)(\J(A)\J(B)\J(C)\J(D)+15) - 3(\I_v(S)+\I_w(S))+\I(S-\{v,w\})$ is invariant under permutations of $A$, $B$, $C$ and $D$, so
$$(4\I_{vw}(S)-\I_v(S))\J(A)\J(B) + (4\I_{vw}(S)-\I_w(S))\J(C)\J(D)$$
has to be maximal over all such permutations. By part (iii) of Lemma~\ref{Lem:I1_I0}, both $4\I_{vw}(S)-\I_v(S)$ and $4\I_{vw}(S)-\I_w(S)$ are positive. So we can apply the rearrangement inequality again to show that either $\J(A),\J(B) \geq \J(C),\J(D)$ or $\J(A),\J(B) \leq \J(C),\J(D)$ (as one could otherwise permute $A,B,C,D$ in such a way that the resulting tree has fewer \icss{}). Lemma~\ref{lem:increasing_m} now yields the statement.
\item[(ii)] For the second part, we first restrict ourselves to the case where $\dd(v,\rr(T))=\dd(w,\rr(T))+1$. For this we use induction on the distance $\dd(w,\rr(T))$. For $\dd(\rr(T),w)=1$, the statement is precisely the second part of Lemma~\ref{Lem:Form2}. Assume now that the claim holds when $\dd(\rr(T),w)=k$ for some $k\geq 1$, and consider the case where $\dd(\rr(T),w)=k+1.$ Let $v'$ and $w'$ be the parents of $v$ and $w$ respectively, so that $\dd(\rr(T),w') = \dd(\rr(T),v') - 1 = k$.
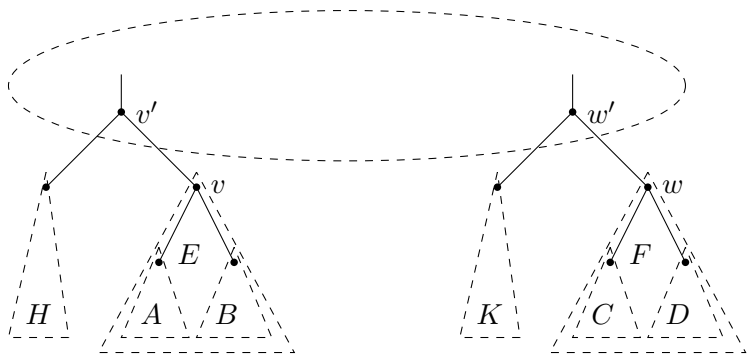
\begin{figure}[htbp]
\centering
 \begin{tikzpicture}[scale=1]
	\draw (-1,-1)--(0,0)--(1,-1)--(1.5,-2);
	\draw (0.5,-2)--(1,-1);
	\draw(0,0)--(0,0.5);
	\draw[dashed] (-1,-0.8)--(-1.5,-3)--(-0.7,-3)--(-1,-0.8);
	\draw[dashed] (0.5,-1.8)--(0,-3)--(0.9,-3)--(0.5,-1.8);
	\draw[dashed] (1.5,-1.8)--(1,-3)--(2,-3)--(1.5,-1.8);
	\draw[dashed] (1,-0.8)--(-0.3,-3.2)--(2.3,-3.2)--(1,-0.8);
	\node[fill=black,label=right:{$w'$},circle,inner sep=1pt] () at (0,0) {};
	\node[fill=black,circle,inner sep=1pt] () at (-1,-1) {};
	\node[fill=black,label=right:{$w$},circle,inner sep=1pt] () at (1,-1) {};
	\node[fill=black,circle,inner sep=1pt] () at (1.5,-2) {};
	\node[fill=black,circle,inner sep=1pt] () at (0.5,-2) {};
	\node[fill=white,label=below:{$K$},circle,inner sep=0pt] () at (-1.1,-2.4) {};
	\node[fill=white,label=below:{$C$},circle,inner sep=0pt] () at (0.4,-2.4) {};
	\node[fill=white,label=below:{$D$},circle,inner sep=0pt] () at (1.4,-2.4) {};
	\node[fill=white,label=below:{$F$},circle,inner sep=0pt] () at (0.9,-1.6) {};
	\draw (-6-1,-1)--(-6+0,0)--(-6+1,-1)--(-6+1.5,-2);
	\draw (-6+0.5,-2)--(-6+1,-1);
	\draw(-6+0,0)--(-6+0,0.5);
	\draw[dashed] (-6-1,-0.8)--(-6-1.5,-3)--(-6-0.7,-3)--(-6-1,-0.8);
	\draw[dashed] (-6+0.5,-1.8)--(-6+0,-3)--(-6+0.9,-3)--(-6+0.5,-1.8);
	\draw[dashed] (-6+1.5,-1.8)--(-6+1,-3)--(-6+2,-3)--(-6+1.5,-1.8);
	\draw[dashed] (-6+1,-0.8)--(-6-0.3,-3.2)--(-6+2.3,-3.2)--(-6+1,-0.8);
	\node[fill=black,label=right:{$v'$},circle,inner sep=1pt] () at (-6+0,0) {};
	\node[fill=black,circle,inner sep=1pt] () at (-6-1,-1) {};
	\node[fill=black,label=right:{$v$},circle,inner sep=1pt] () at (-6+1,-1) {};
	\node[fill=black,circle,inner sep=1pt] () at (-6+1.5,-2) {};
	\node[fill=black,circle,inner sep=1pt] () at (-6+0.5,-2) {};
	\node[fill=white,label=below:{$H$},circle,inner sep=0pt] () at (-6-1.1,-2.4) {};
	\node[fill=white,label=below:{$A$},circle,inner sep=0pt] () at (-6+0.4,-2.4) {};
	\node[fill=white,label=below:{$B$},circle,inner sep=0pt] () at (-6+1.4,-2.4) {};
	\node[fill=white,label=below:{$E$},circle,inner sep=0pt] () at (-6+0.9,-1.6) {};
	\draw[dashed] (-3,0.35) ellipse (4.5cm and 1cm);
	\end{tikzpicture}
\caption{Further decomposition of $T$ in the proof of Lemma~\ref{Lem:GenComp}.}\label{fig:GenComp}
\end{figure}
If $w'$ is the parent of $v'$, then we are immediately done by applying part (ii) of Lemma~\ref{Lem:Form2} to the tree rooted at $w'$. Otherwise, we can apply the induction hypothesis to $v'$ and $w'$ (see Figure~\ref{fig:GenComp} for an illustration), which shows that
$$
\max\{|H|,|E|\}\leq\min\{|K|,|F|\},
$$
and thus $|E|\leq |F|$. If $|A|=|B|=|C|=|D|$, then we are done. Otherwise, we must have
$$
\min\{|A|,|B|\}<\max\{|C|,|D|\},
$$
since
$$|A| + |B| + 1 = |E| \leq |F| = |C|+|D| + 1.$$
Without loss of generality, we can assume that $|A|<|D|$. But then, by part (i), we must have
$$
\max\{|A|,|B|\}\leq\min\{|C|,|D|\}
$$
as desired.

If $\dd(\rr(T),v) > \dd(\rr(T),w) + 1$, we can argue as follows: let $v'$ be the ancestor of $v$ in $T$ for which $\dd(\rr(T),v') = \dd(\rr(T),w) + 1$. The branch rooted at $v'$ is minimal, so by the second part of Lemma~\ref{Lem:g8} it has exactly two children unless its order is less than $8$. If $v'$ has precisely two children, then we know that the desired inequality holds for the branches rooted at the children of $v'$ and $w$. But since $A$ and $B$ are contained in one of the branches rooted at a child of $v'$, it also holds for them. If the branch rooted at $v'$ has fewer than eight vertices, then (as can be seen from Table~\ref{Tab:Min_Small}) we must have $|A| = |B| = 1$, and the inequality becomes trivial.
\end{itemize}
\end{proof}

We already know from Lemma~\ref{Lem:g8} that the root degree of a minimal tree with at least eight vertices is necessarily $2$. The following lemma strengthens this further.
\begin{lem}
\label{Lem:Form0}
If $T$ is a minimal tree with $|T|\geq 18$, then both the root and its children have exactly two children.
\end{lem}
\begin{proof}
For $18 \leq |T| < 24$, this can be checked directly, see Table~\ref{Tab:Min_Small}. So assume that $|T| \geq 24$. Clearly, at least one of the root branches of $T$ has eight or more vertices, so its root has degree~$2$ by Lemma~\ref{Lem:g8}. So $T$ is of the form $[[A,B],C]$, as in Figure~\ref{fig:ABC_decomp},
where $\max\{|A|,|B|\}\leq |C|$ in view of Lemma~\ref{Lem:Form2}. If $|C| \leq 7$, then
$$|T| = |A|+|B|+|C| + 2 \leq 3|C| + 2 \leq 23,$$
which is a contradiction. So $|C|$ has at least eight vertices, which means that its root degree is $2$ by Lemma~\ref{Lem:g8}. This concludes the proof.
\end{proof}

\subsection{Heights}

The height $\h(T)$ of a rooted tree $T$ is the length of the longest path from the root to a leaf. In our next lemma, we show a relation between the heights and the orders of branches. It will be useful for the formulation of this lemma, and also in the following, to use the following convention: a minimal tree is said to have \emph{standard form} if every branch of order $7$ has the shape $[[\bullet,\bullet],[\bullet,\bullet]]$. 

\begin{lem}
\label{Lem:Branch_height}
Let $A$ and $B$ be branches of a minimal tree $T$ in standard form. If $|A|\leq |B|$, then $\h(A)\leq \h(B)$.
\end{lem}
\begin{proof}
Recall that $A$ and $B$ are necessarily minimal trees themselves. We use induction on $|B|$ to prove the first part. If $|A| \leq |B| \leq 8$, we can simply check all possible cases for $A$ and $B$, see Table~\ref{Tab:Min_Small}. 
Now assume that the statement holds whenever $|B|\leq k$ for some $k\geq 8$, and consider the case where $|B|=k+1$. 
In view of Lemma~\ref{Lem:g8}, we know that $\h(B)\geq 2$ since $B$ cannot be a star. So if $|A|\leq 7$, then $\h(A)\leq 2\leq \h(B)$, and we are done. If $|A|\geq 8$, then both $A$ and $B$ have root degree $2$, and $T$ has the form shown in Figure~\ref{Fig:A1A2} (unless $A$ is contained in $B$, in which case the statement becomes trivial).
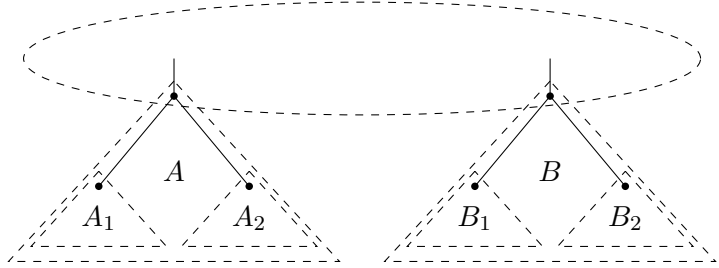
\begin{figure}[htbp]
\centering
\begin{tikzpicture}[scale=1]
    \draw (0,0)--(0,0.5);
	\draw (-1,-1.2)--(0,0)--(1,-1.2);
	\draw[dashed] (-1,-1)--(-1.9,-2)--(-0.1,-2)--(-1,-1);
	\draw[dashed] (1,-1)--(1.9,-2)--(0.1,-2)--(1,-1);
    	\draw[dashed] (0,0.2)--(-2.2,-2.2)--(2.2,-2.2)--(0,0.2);
    \node[fill=white,label=below:{$A$},circle,inner sep=0pt] () at (0,-0.7) {};
    \node[fill=black,circle,inner sep=1pt] () at (0,0) {}; 
    \node[fill=black,circle,inner sep=1pt] () at (-1,-1.2) {};
    \node[fill=black,circle,inner sep=1pt] () at (1,-1.2) {};
    \node[fill=white,label=below:{$A_1$},circle,inner sep=0pt] () at (-1,-1.3) {};
    \node[fill=white,label=below:{$A_2$},circle,inner sep=0pt] () at (1,-1.3) {};
    \draw (5+0,0)--(5+0,0.5);
	\draw (5-1,-1.2)--(5+0,0)--(5+1,-1.2);
	\draw[dashed] (5-1,-1)--(5-1.9,-2)--(5-0.1,-2)--(5-1,-1);
	\draw[dashed] (5+1,-1)--(5+1.9,-2)--(5+0.1,-2)--(5+1,-1);
	\draw[dashed] (5,0.2)--(5-2.2,-2.2)--(5+2.2,-2.2)--(5,0.2);
    \node[fill=black,circle,inner sep=1pt] () at (5+0,0) {}; 
    \node[fill=black,circle,inner sep=1pt] () at (5-1,-1.2) {};
    \node[fill=black,circle,inner sep=1pt] () at (5+1,-1.2) {};
    \node[fill=white,label=below:{$B$},circle,inner sep=0pt] () at (5,-0.7) {};
    \node[fill=white,label=below:{$B_1$},circle,inner sep=0pt] () at (5-1,-1.3) {};
    \node[fill=white,label=below:{$B_2$},circle,inner sep=0pt] () at (5+1,-1.3) {};
    \draw[dashed] (2.5,0.5) ellipse (4.5cm and 0.75cm);
    \end{tikzpicture} 
\caption{Proof of Lemma~\ref{Lem:Branch_height}.}
\label{Fig:A1A2}
\end{figure}

By Lemma~\ref{Lem:GenComp}, we have 
$$
\max\{|A_1|,|A_2|\} \leq \min\{|B_1|,|B_2|\},
$$
since the other inequality $\max\{|B_1|,|B_2|\}\leq \min\{|A_1|,|A_2|\}$ can be ruled out as $|A| \leq |B|$ (unless $|A_1|  = |A_2| = |B_1| = |B_2|$, in which case both inequalities hold). By the induction hypothesis we have
$$
\max\{\h(A_1),\h(A_2)\}  \leq \min\{\h(B_1),\h(B_2)\},
$$
which in turn implies that
$$\h(A) = 1 + \max \{\h(A_1),\h(A_2)\} \leq 1 + \min\{\h(B_1),\h(B_2)\} \leq 1 + \max\{\h(B_1),\h(B_2)\} = \h(B).$$
This completes the proof.
\end{proof}

\subsection{Excluded branches}

By Lemma~\ref{Lem:GenComp}, it is clear that a minimal tree cannot contain two copies of a branch $H=[A,B]$ where $|A|>|B|$. The following two lemmas determine further forbidden occurrences of branches in minimal trees.

\begin{lem}
\label{Lem:Forb_Br}
Let $H$ be a branch of a minimal tree $T$. We have
\begin{itemize}
\item[(i)] $|H| \neq 6$ unless $H = T$, and
\item[(ii)] if $|H|=3$, then either $H = T$ or $H$ is part of a branch of order $7$ or $8$.
\end{itemize}
\end{lem}
\begin{proof}
We prove both statements by induction. For $|T|\leq 20$, we simply check the list of minimal trees in Table~\ref{Tab:Min_Small}. Assume now that every minimal tree of order at most $k$, for some $k\geq 20$, does not contain a branch of order $6$ as a proper subtree, and that each of the branches of order $3$ is either the whole tree or part of a branch of order $7$ or $8$.  Now consider a minimal tree $T$ of order $k+1 \geq 21$. By Lemma~\ref{Lem:Form0}, it must have the form $T = [A,B] = [[A_1,A_2],[B_1,B_2]]$ shown in Figure~\ref{fig:A1A2B1B2}.
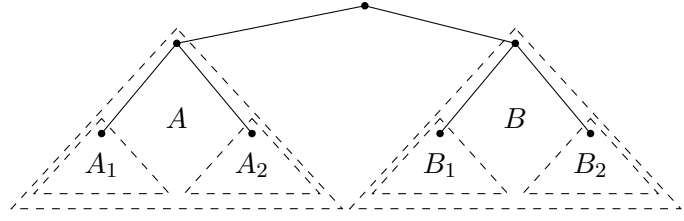
\begin{figure}[htbp]
\centering
\begin{tikzpicture}[scale=1]
    \draw (0,0)--(2.5,0.5);
	\draw (-1,-1.2)--(0,0)--(1,-1.2);
	\draw[dashed] (-1,-1)--(-1.9,-2)--(-0.1,-2)--(-1,-1);
	\draw[dashed] (1,-1)--(1.9,-2)--(0.1,-2)--(1,-1);
    	\draw[dashed] (0,0.2)--(-2.2,-2.2)--(2.2,-2.2)--(0,0.2);
    \node[fill=white,label=below:{$A$},circle,inner sep=0pt] () at (0,-0.7) {};
    \node[fill=black,circle,inner sep=1pt] () at (2.5,0.5) {}; 
    \node[fill=black,circle,inner sep=1pt] () at (0,0) {}; 
    \node[fill=black,circle,inner sep=1pt] () at (-1,-1.2) {};
    \node[fill=black,circle,inner sep=1pt] () at (1,-1.2) {};
    \node[fill=white,label=below:{$A_1$},circle,inner sep=0pt] () at (-1,-1.3) {};
    \node[fill=white,label=below:{$A_2$},circle,inner sep=0pt] () at (1,-1.3) {};
      \draw (4.5+0,0)--(2.5,0.5);
	\draw (4.5-1,-1.2)--(4.5+0,0)--(4.5+1,-1.2);
	\draw[dashed] (4.5-1,-1)--(4.5-1.9,-2)--(4.5-0.1,-2)--(4.5-1,-1);
	\draw[dashed] (4.5+1,-1)--(4.5+1.9,-2)--(4.5+0.1,-2)--(4.5+1,-1);
	\draw[dashed] (4.5,0.2)--(4.5-2.2,-2.2)--(4.5+2.2,-2.2)--(4.5,0.2);
    \node[fill=black,circle,inner sep=1pt] () at (4.5+0,0) {}; 
    \node[fill=black,circle,inner sep=1pt] () at (4.5-1,-1.2) {};
    \node[fill=black,circle,inner sep=1pt] () at (4.5+1,-1.2) {};
    \node[fill=white,label=below:{$B$},circle,inner sep=0pt] () at (4.5,-0.7) {};
    \node[fill=white,label=below:{$B_1$},circle,inner sep=0pt] () at (4.5-1,-1.3) {};
    \node[fill=white,label=below:{$B_2$},circle,inner sep=0pt] () at (4.5+1,-1.3) {};
    \end{tikzpicture}
\caption{Decomposition $T = [A,B] = [[A_1,A_2],[B_1,B_2]]$.}\label{fig:A1A2B1B2}
\end{figure}
Without loss of generality, we can assume $|A|\geq |B|$, and by Lemma~\ref{Lem:Form2} we know that $|B| \geq \max\{|A_1|,|A_2|\}$.
Now it is impossible that $|B|\leq 6$, since this would imply
$$|T| = |A_1|+|A_2|+|B|+2 \leq 3|B|+2 \leq 20.$$
Thus $|A| \geq |B| \geq 7$. So if $T$ has a $6$-vertex branch, then this branch has to lie inside $A$ or $B$, but is not equal to either of the two. This contradicts the induction hypothesis. Likewise, if $T$ has a branch $H$ of order $3$, then it has to be a branch in $A$ or $B$. So again, by the induction hypothesis, it has to be part of a branch with $7$ or $8$ vertices.
\end{proof}
\begin{lem}
\label{Lem:Forb_Br_1}
Let $T$ be a minimal tree, and let $A$ and $B$ be two disjoint branches, so that $T$ can be represented as $T=[A:_v\hspace{-0.14cm}S_w\hspace{-0.14cm}:B]$. Assume without loss of generality that $|A| \geq |B|$. Then the pair $(|A|,|B|)$ is none of the following:
$$(7,7),\ (8,8),\ (11,11),\ (7,5),\ (8,5),\ (8,7),\ (11,10),\ (16,10),\ (15,11),\ (16,11).$$

\end{lem}
\begin{proof}
For each of the pairs that are listed, we provide a transformation that yields a new tree $T'$ with fewer \icss{}, which is a contradiction. 
Let us show this in the case $|A|=|B|=7$ first.
The new tree $T'$ is obtained by the transformation shown in Figure~\ref{fig:Forb_Br_trans} (note that it is irrelevant for the number of \icss{} which of the two possible shapes a minimal branch of order $7$ has). In words, $A$ is replaced by a minimal tree of order $9$, $B$ is replaced by a minimal tree of order $5$. 
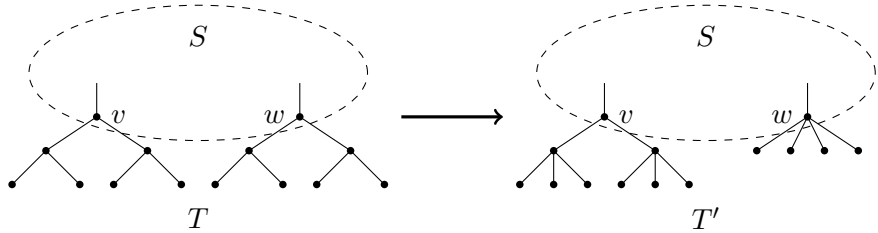
\begin{figure}[htbp]
\centering
\begin{tikzpicture}[scale=0.9]
    \draw[dashed] (0,0.65) ellipse (2.5cm and 1cm);
    \draw (-1.5,0.5)--(-1.5,0)--(-0.75,-0.5)--(-0.25,-1);
    \draw (-1.5,0)--(-2.25,-0.5)--(-2.75,-1);
    \draw (-0.75,-0.5)--(-1.25,-1);
    \draw (-2.25,-0.5)--(-1.75,-1);
    \node[fill=black,label=right:{$v$},circle,inner sep=1pt] () at (-1.5,0) {};
    \node[fill=black,circle,inner sep=1pt] () at (-0.75,-0.5) {};
    \node[fill=black,circle,inner sep=1pt] () at (-0.25,-1) {};
    \node[fill=black,circle,inner sep=1pt] () at (-1.25,-1) {};
    \node[fill=black,circle,inner sep=1pt] () at (-2.25,-0.5) {};
    \node[fill=black,circle,inner sep=1pt] () at (-1.75,-1) {};
    \node[fill=black,circle,inner sep=1pt] () at (-2.75,-1) {};
    \draw (3-1.5,0.5)--(3-1.5,0)--(3-0.75,-0.5)--(3-0.25,-1);
    \draw (3-1.5,0)--(3-2.25,-0.5)--(3-2.75,-1);
    \draw (3-0.75,-0.5)--(3-1.25,-1);
    \draw (3-2.25,-0.5)--(3-1.75,-1);
    \node[fill=black,label=left:{$w$},circle,inner sep=1pt] () at (3-1.5,0) {};
    \node at (0,1.2) {$S$};
    \node at (0,-1.5) {$T$};
    \node[fill=black,circle,inner sep=1pt] () at (3-0.75,-0.5) {};
    \node[fill=black,circle,inner sep=1pt] () at (3-0.25,-1) {};
    \node[fill=black,circle,inner sep=1pt] () at (3-1.25,-1) {};
    \node[fill=black,circle,inner sep=1pt] () at (3-2.25,-0.5) {};
    \node[fill=black,circle,inner sep=1pt] () at (3-1.75,-1) {};
    \node[fill=black,circle,inner sep=1pt] () at (3-2.75,-1) {};
    \draw[->, very thick] (3,0)--(4.5,0);
    \draw[dashed] (7.5+0,0.65) ellipse (2.5cm and 1cm);
    \draw (7.5-1.5,0.5)--(7.5-1.5,0)--(7.5-0.75,-0.5)--(7.5-0.25,-1);
    \draw (7.5-1.5,0)--(7.5-2.25,-0.5);
    \draw (7.5-1.25,-1)--(7.5-0.75,-0.5)--(7.5-0.75,-1);
    \draw (6-1.25,-1)--(6-0.75,-0.5)--(6-0.75,-1);
    \draw (6-0.25,-1)--(6-0.75,-0.5);
    \node[fill=black,label=right:{$v$},circle,inner sep=1pt] () at (7.5-1.5,0) {};
    \node[fill=black,circle,inner sep=1pt] () at (6-1.25,-1) {};
    \node[fill=black,circle,inner sep=1pt] () at (6-0.75,-1) {};
    \node[fill=black,circle,inner sep=1pt] () at (6-0.25,-1) {};
    \node[fill=black,circle,inner sep=1pt] () at (7.5-0.75,-0.5) {};
    \node[fill=black,circle,inner sep=1pt] () at (7.5-0.25,-1) {};
    \node[fill=black,circle,inner sep=1pt] () at (7.5-2.25,-0.5) {};
    \node[fill=black,circle,inner sep=1pt] () at (7.5-1.25,-1) {};
    \node[fill=black,circle,inner sep=1pt] () at (7.5-0.75,-1) {};
    \draw (7.5+3-1.5,0.5)--(7.5+3-1.5,0)--(7.5+3-0.75,-0.5);
    \draw (7.5+3-1.75,-0.5)--(7.5+3-1.5,0)--(7.5+3-2.25,-0.5);
    \draw (7.5+3-1.25,-0.5)--(7.5+3-1.5,0);
    \node[fill=black,label=left:{$w$},circle,inner sep=1pt] () at (7.5+3-1.5,0) {};
    \node at (7.5,1.2) {$S$};
    \node at (7.5,-1.5) {$T'$};
    \node[fill=black,circle,inner sep=1pt] () at (7.5+3-0.75,-0.5) {};
    \node[fill=black,circle,inner sep=1pt] () at (7.5+3-1.75,-0.5) {};
    \node[fill=black,circle,inner sep=1pt] () at (7.5+3-1.25,-0.5) {};
    \node[fill=black,circle,inner sep=1pt] () at (7.5+3-2.25,-0.5) {};
    \end{tikzpicture} 
\caption{Constructing a tree $T'$ with fewer \icss{}.}\label{fig:Forb_Br_trans}
\end{figure}
We use Lemma~\ref{Lem:S_A} to express $\I(T)$ and $\I(T')$:
\begin{align*}
\I(T) &=61^2\I_{vw}(S)+61(\I_{v}(S)-\I_{vw}(S))+61(\I_{w}(S)-\I_{vw}(S))+\I(S-\{v,w\}) \\
&= 3599\I_{vw}(S) + 61\I_{v}(S) + 61\I_{w}(S) +\I(S-\{v,w\})
\end{align*}
and
\begin{align*}
\I(T') &=166 \cdot 20\I_{vw}(S)+166(\I_{v}(S)-\I_{vw}(S))+20(\I_{w}(S)-\I_{vw}(S))+\I(S-\{v,w\}) \\
&= 3134\I_{vw}(S) + 166\I_{v}(S) + 20\I_{w}(S) +\I(S-\{v,w\}),
\end{align*}
so
$$
\I(T)-\I(T') =465\I_{vw}(S)-105\I_v(S)+41\I_w(S) >0.
$$
Here, we are using the relation $3\I_{vw}(S)\geq \I_v(S)$ from Lemma~\ref{Lem:I1_I0}, part (iii). This proves that $|A| = |B| = 7$ is in fact impossible.

We proceed in a similar way in all other cases. If $|A| = |B| = 8$, then we can assume without loss of generality (interchanging the roles of $A$ and $B$ if necessary) that $\I_v(S)\leq \I_w(S)$. The new tree $T'$ is constructed by replacing $A$ and $B$ by minimal trees of order $9$ and $7$ respectively. Here, we obtain
$$\I(T) - \I(T') = 100 \I_{vw}(S) - 65\I_v(S)+40 \I_w(S) = 100 \I_{vw}(S) - 25 \I_v(S) + 40(\I_w(S) - \I_v(S)) >0,$$
exploiting again the relation $3\I_{vw}(S)\geq \I_v(S)$.

The remaining cases are summarised in Table~\ref{table:Forb_Br_1}: in each case, $A$ and $B$ are replaced by minimal trees $A'$ and $B'$ of different orders to obtain a new tree $T'$. Note that we may assume in each case that $I_v(S) \leq I_w(S)$: if not, switching $A$ and $B$ yields a tree $\bar{T}$ with
$$\I(T) - \I(\bar{T}) = (\I_v(S) - \I_w(S))(\I(A)-\I(B)) > 0,$$
which is a contradiction. In the case $|A|=|B| = 11$, we can assume $I_v(S) \leq I_w(S)$ by symmetry. This and the inequality $3\I_{vw}(S) \geq \I_v(S)$ from Lemma~\ref{Lem:I1_I0} suffice to show that $\I(T) > \I(T')$ in each case.
\begin{table}[htbp]
\centering
\begin{tabular}{|c|c|c|c|l|}
\hline
$|A|$ & $|B|$ & $|A'|$ & $|B'|$ & $\I(T) - \I(T')$ \\
\hline
11 & 11 & 13 & 9 & $10125\I_{vw}(S) - 855\I_v(S) + 315\I_w(S)$ \\
7 & 5 & 8 & 4 & $140\I_{vw}(S) - 40\I_v(S) + 9\I_w(S)$ \\
8 & 5 & 9 & 4 & $250\I_{vw}(S) - 65\I_v(S) + 9\I_w(S)$ \\
8 & 7 & 9 & 6 & $225\I_{vw}(S) - 65\I_v(S) + 25\I_w(S)$ \\
11 & 10 & 12 & 9 & $885\I_{vw}(S) - 335\I_v(S) + 117\I_w(S)$ \\
16 & 10 & 17 & 9 & $21312\I_{vw}(S) - 4314\I_v(S) + 117\I_w(S)$ \\
15 & 11 & 17 & 9 & $26460\I_{vw}(S) - 6888\I_v(S) + 315\I_w(S)$ \\
16 & 11 & 18 & 9 & $153180 \I_{vw}(S) - 11034 \I_v(S) + 315 \I_w(S)$ \\
\hline
\end{tabular}

\medskip

\caption{Replacements in the proof of Lemma~\ref{Lem:Forb_Br_1}.}\label{table:Forb_Br_1}
\end{table}

This completes the proof of our lemma.
\end{proof}

Lemma~\ref{Lem:Forb_Br} and Lemma~\ref{Lem:Forb_Br_1} imply the following statement on leaves in minimal trees. Recall here that a minimal tree is said to have standard form if every branch of order $7$ has the shape $[[\bullet,\bullet],[\bullet,\bullet]]$. 

\begin{lem}\label{Lem:Br_3_5}
Let $T$ be a minimal tree of order $n \geq 7$ in standard form.
\begin{itemize}
\item[(i)] Every leaf of $T$ is contained in a branch of order $3$, $4$ or $5$.
\item[(ii)] Branches of order $3$ and $5$ do not occur simultaneously in $T$.
\item[(iii)] At most two branches of order $3$ occur in $T$. If there are exactly two, then they are part of a branch of order $7$.
\end{itemize}
\end{lem}

\newpage 

\begin{proof}\
\begin{itemize}
\item[(i)] By Lemma~\ref{Lem:Form0}, the parent of a leaf is root of a branch of order at most $17$. Looking at Table~\ref{Tab:Min_Small}, we see that it is in fact root of a branch of order at most $7$. Since $T$ is assumed to have standard form, order $7$ is also excluded, as is order $6$ by Lemma~\ref{Lem:Forb_Br} and order $2$ by Lemma~\ref{Lem:g8}. Consequently, the order of the branch rooted at a leaf's parent is $3$, $4$ or $5$.
\item[(ii)] A branch of order $3$ can only occur as part of a branch of order $7$ or $8$ by Lemma~\ref{Lem:Forb_Br}. Since the pairs $(7,5)$ and $(8,5)$ are forbidden in Lemma~\ref{Lem:Forb_Br_1}, it can therefore not occur simultaneously with a branch of order $5$.
\item[(iii)] Again, we use the fact that every branch of order $3$ is part of a branch of order $7$ or $8$. Since the pairs $(7,7)$, $(8,8)$ and $(8,7)$ all belong to the forbidden pairs of Lemma~\ref{Lem:Forb_Br_1}, there can be only at most one branch of order $7$ or $8$, thus at most two of order $3$ (and if there are two, then they form part of the same branch of order $7$).
\end{itemize}
\end{proof}

The following lemma looks specifically at branches of order $5$.

\begin{lem}
\label{Lem:Place_of_5}
If a minimal tree $T$ of order $n\geq 6$ has a $5$-vertex branch, then it must be part of a branch of order $10$, $11$, $15$ or $16$. 
\end{lem}
\begin{proof}
The proof is similar to that of part (i) of Lemma~\ref{Lem:Forb_Br}, by induction on the order $n$. A direct check of Table~\ref{Tab:Min_Small} shows that the property holds if $|T|\leq 50$. If $|T|\geq 51$, then by Lemma~\ref{Lem:Form0} and Lemma~\ref{Lem:Form2}, the tree can be expressed as $T=[[A_1,A_2],B]$ for some minimal rooted trees $A_1,A_2$ and $B$, where $|A_1| \leq |A_2| \leq |B| \leq |[A_1,A_2]|$. Since $|A_1| + |A_2| + |B| + 2 = |T| \geq 51$, we must have $|B|\geq 17$. Any $5$-vertex branch of $T$ has to be a proper subtree of $B$ or $[A_1,A_2]$, and by the induction hypothesis it has to be part of a branch of order $10$, $11$, $15$ or $16$. 
\end{proof}
\begin{lem}
\label{Lem:5_5}
A minimal tree cannot have more than two branches of order $5$.
\end{lem}
\begin{proof}
We know from Lemma~\ref{Lem:Place_of_5} that every $5$-vertex branch of a minimal tree $T$ has to be inside some branch of order $10$, $11$, $15$ or $16$. If there are more than two branches of order $5$, then there have to be at least two branches whose orders are in the set $\{10,11,15,16\}$. Since the minimal trees of order $10$, $15$ and $16$ have root branches of unequal orders, they cannot be repeated in a minimal tree by Lemma~\ref{Lem:GenComp}. By the same lemma, branches of order $15$ and $16$ cannot occur simultaneously either. By Lemma~\ref{Lem:Forb_Br_1}, the pairs $(11,11)$, $(11,10)$, $(15,11)$, $(16,10)$ and $(16,11)$ are also all excluded. The only possible remaining pair of branches with orders in $\{10,11,15,16\}$ is $(15,10)$, which however only gives us two $5$-vertex branches.
\end{proof}

\section{The main result}\label{Sec:Main}

In this section, we finally reach a full characterisation of minimal trees. To this end, we first need two more technical lemmas. Recall that the height $\h(T)$ of a rooted tree $T$ is the greatest distance from the root to a leaf. The $i$-th level of vertices (where $0 \leq i \leq \h(T)$) consists of all vertices whose distance from the root is exactly $i$.

\begin{lem}
\label{Lem:Compl}
Let $T$  be a minimal tree with $n\geq 7$ vertices in standard form, and let $A$ and $B$ be its root branches, so that $T = [A,B]$. Then the following statements hold:
\begin{itemize}
\item[(i)] $|\h(A)-\h(B)|\leq 1$.
\item[(ii)] The distances of any two leaves to the root of $T$ differ by at most $1$. In other words, the leaves of $T$ can only be at the last two levels. 
\end{itemize}
\end{lem}

\begin{proof}\
\begin{itemize}
\item[(i)] For $n \leq 17$, the claim is easily verified directly by means of Table~\ref{Tab:Min_Small}. If $n \geq 18$, then
both the root and the root's children have exactly two children by Lemma~\ref{Lem:Form0}. So $T$ is of the form $T = [[A_1,A_2],[B_1,B_2]]$, as in Figure~\ref{fig:A1A2B1B2}. Moreover, we may assume without loss of generality that
$$
\min\{|A_1|, |A_2|\}\geq \max\{|B_1|, |B_2|\}
$$
by Lemma~\ref{Lem:GenComp}, and consequently that
$$
\min\{\h(A_1), \h(A_2)\}\geq \max\{\h(B_1), \h(B_2)\}
$$
by Lemma~\ref{Lem:Branch_height}, which implies
\begin{equation}\label{eq:height1}
\h(A)  = \max\{\h(A_1), \h(A_2)\} + 1 \geq \max\{\h(B_1), \h(B_2)\} + 1 = \h(B).
\end{equation}
In addition, we know from part (i) of Lemma~\ref{Lem:Form2} that
\begin{equation}\label{eq:B-size}
\max\{|A_1|, |A_2|\}\leq |B|,
\end{equation}
and thus
\begin{equation}\label{eq:height2}
\h(A) = \max\{\h(A_1), \h(A_2)\} + 1 \leq \h(B) + 1,
\end{equation}
again by Lemma~\ref{Lem:Branch_height}. Part (i) follows from inequalities~\eqref{eq:height1} and~\eqref{eq:height2}.
\item[(ii)] The second part is proven by induction. Again, it is easy to check the statement directly for $n \leq 17$ (in fact, it holds for all $n$, not just $n \geq 7$). Assume that it holds whenever $n\leq k$ for some $k\geq 17$, and let the order of $T$ be $n=k+1$. Again, we know that $T$ must be of the form $T = [A,B] = [[A_1,A_2],[B_1,B_2]]$, and we can assume that $\max\{|A_1|,|A_2|\} \leq |B| \leq |A|$.
Lemma \ref{Lem:Branch_height}, the induction hypothesis and part (i) imply that the leaves of $T$ all lie at the last three levels. To complete the proof of (ii), it is only left to show that there is no leaf at the third level from the bottom.

Assume that there is such a leaf $f$, i.e., a leaf whose distance from the root $\rr(T)$ is $\h(T) - 2$. Note that it has to lie in branch $B$. Let $w$ be its parent, and let $w'$ be $w$'s parent. From Lemma~\ref{Lem:Br_3_5}, we know that the branch rooted at $w$ has order $3$, $4$ or $5$. If $w' = \rr(T)$, then this is the branch $B$, i.e., $3 \leq |B| \leq 5$. But then,~\eqref{eq:B-size} shows that
$$|T| = |A_1| + |A_2| + |B| + 2 \leq 3|B| + 2 \leq 17,$$
which is a contradiction. So we may assume that $w'$ is not the root. By the induction hypothesis, the branch rooted at $w'$ has no leaf at level $\h(T)$. 

Now pick any vertex $v$ at the same level as $f$ that has a descendant at level $\h(T)$, and let $v'$ be its parent (which is at the same level as $w$). Note that $v'$ cannot be in the branch rooted at $w'$. The branches rooted at $v'$ and $w'$ both have height $2$ or greater, thus order at least $7$ (since minimal trees of order at most $5$ are stars and branches of order $6$ do not occur by Lemma~\ref{Lem:Forb_Br}). If the branch rooted at $v'$ (resp. $w'$) has order $8$ or greater, then $v'$ (resp. $w'$) has exactly two children by 
Lemma~\ref{Lem:g8}. This is also true if the order of the branch is $7$, as we are assuming $T$ to be in standard form. Thus we are in a situation where part (ii) of Lemma~\ref{Lem:GenComp} applies (with $v'$ and $w'$ taking the roles of $v$ and $w$ in that lemma). So the order of the branch rooted at $v$ is less than or equal to the order of the branch rooted at $w$. But this is impossible, since the former has height $2$ and thus order at least $7$, while the order of the latter is at most $5$. This completes the proof of part (ii).
\end{itemize}

\end{proof}

Our final preparatory lemma shows that the residue class of the order modulo $5$ plays an essential role for the structure of a minimal tree.

\begin{lem}
\label{Lem:main}
If $T$ is a minimal tree with $n\geq 8$ vertices in standard form, then the following holds:
\begin{itemize}
\item[(i)]if $n\equiv 0 \mod 5$, then $T$ has exactly one branch of order $5$, and every leaf of $T$ either belongs to this branch or a branch of order $4$.
\item[(ii)]if $n\equiv 1 \mod 5$, then $T$ has exactly two branches of order $5$, and every leaf of $T$ either belongs to one of these branches or a branch of order $4$.
\item[(iii)]if $n\equiv 2 \mod 5$, then $T$ has a $7$-vertex branch, and every leaf of $T$ either belongs to this branch or a branch of order $4$.
\item[(iv)]if $n\equiv 3 \mod 5$, then $T$ has exactly one branch of order $3$, and every leaf of $T$ either belongs to this branch or a branch of order $4$.
\item[(v)]if $n\equiv 4 \mod 5$, then every leaf of $T$ belongs to a branch of order $4$.
\end{itemize}
\end{lem}
\begin{proof}
We know the following from Lemma~\ref{Lem:Br_3_5} and Lemma~\ref{Lem:5_5}:
\begin{itemize}
\item Every leaf belongs to a branch of order $3$, $4$ or $5$.
\item There are at most two branches of order $3$ (if exactly two, they belong to a branch of order $7$), and at most two branches of order $5$.
\item Branches of order $3$ and $5$ do not occur simultaneously.
\end{itemize}
Moreover, there are no branches of order $2$ (by Lemma~\ref{Lem:g8}) or $6$ (by Lemma~\ref{Lem:Forb_Br}), and all vertices that are roots of larger branches have exactly two children (if the order is $7$, then this holds because $T$ is in standard form, otherwise by Lemma~\ref{Lem:g8}). So if we remove all leaves, the resulting tree $T'$ is binary, i.e., every vertex is a leaf or has exactly two children. This means that it has $m-1$ internal vertices (non-leaves) and $m$ leaves for some positive integer $m$. Based on how many branches of order $3$ or $5$ occur, we have five scenarios:
\begin{itemize}
\item One branch of order $5$: one of the leaves of $T'$ has four children in $T$, all others three. This gives us a total of $n = (m-1)+m+4 + (m-1) \cdot 3 = 5m$ vertices.
\item Two branches of order $5$: two of the leaves of $T'$ have four children in $T$, all others three. This gives us a total of $n = (m-1)+m+2 \cdot 4 + (m-2) \cdot 3 = 5m+1$ vertices.
\item Two branches of order $3$ (thus one branch of order $7$): two of the leaves of $T'$ have two children in $T$, all others three. This gives us a total of $n = (m-1)+m+2 \cdot 2 + (m-2) \cdot 3 = 5m-3$ vertices.
\item One branch of order $3$: one of the leaves of $T'$ has two children in $T$, all others three. This gives us a total of $n = (m-1)+m+2 + (m-1) \cdot 3 = 5m-2$ vertices.
\item No branches of order $3$ or $5$: all leaves of $T'$ have three children in $T$. This gives us a total of $n = (m-1)+m + m \cdot 3 = 5m-1$ vertices.
\end{itemize}
We see that each of them corresponds to exactly one of the congruence classes modulo $5$, which proves the statement.
\end{proof}

Now we can finally give a complete description of minimal trees. Recall that the $i$-th level of vertices in a rooted tree consists of all vertices whose distance from the root is $i$. For a rooted tree embedded in the plane, we define the canonical order of vertices in the following way: we order first by level (i.e., all vertices at level $i$ come before all vertices at level $j$ if $i < j$), and from left to right within each level.

\begin{thm}\label{thm:main}
Let $n$ be a positive integer. If $n > 6$, then a minimal tree of order $n$ can be constructed as follows:
\begin{itemize}
\item Let $h$ be the unique integer for which $5 \cdot 2^{h-1} + 1 < n \leq 5 \cdot 2^h + 1$, and let $m$ be the nearest integer to $\frac{n+1}{5}$, so that $n \in \{5m-3,5m-2,5m-1,5m,5m+1\}$.
\item Start from a complete binary tree of height $h-1$, i.e., a tree in which all non-leaves have exactly two children, and all leaves are at level $h-1$. This tree has $2^h -1$ vertices, of which $2^{h-1}$ are leaves.
\item Replace each of the $m-2^{h-1}$ leftmost leaves by a $9$-vertex rooted tree $[[\bullet,\bullet,\bullet],[\bullet,\bullet,\bullet]]$ (i.e., attach two children, and three children to each of them). Note here that $m > 2^{h-1}$ by our choice of $h$ and $m$.
\item Replace each of the remaining $2^h-m$ leaves by a $4$-vertex rooted tree $[\bullet,\bullet,\bullet]$ (i.e., attach three children). The resulting tree has $5m-1$ vertices.
\item Consider five subcases according to the residue class of $n$ modulo $5$:
\begin{itemize}
\item If $n = 5m-1$, then we are done.
\item If $n = 5m-2$, remove the last leaf (according to the canonical order).
\item If $n = 5m-3$, remove one child each from the last two vertices in the canonical order that have leaves.
\item If $n = 5m$, add one child to the first vertex in the canonical order that has leaf children.
\item If $n = 5m+1$, add one child to each of the first two vertices in the canonical order that have leaf children.
\end{itemize}
\end{itemize}
If $n \not\equiv 2 \mod 5$, then this is the only minimal tree. Otherwise, there is a second minimal tree obtained from the tree described above by replacing the only branch of the form $[[\bullet,\bullet],[\bullet,\bullet]]$ by $[[\bullet,\bullet,\bullet],\bullet,\bullet]$, and there are no other minimal trees.
\end{thm}

\begin{proof}
Let $T$ be a minimal tree with $n$ vertices in standard form. By now, we know that all leaves belong to branches of the form $[\bullet,\bullet]$, $[\bullet,\bullet,\bullet]$ or $[\bullet,\bullet,\bullet,\bullet]$, and that all non-leaves that are not root of such a branch have exactly two children. We draw $T$, starting from the root, in such a way that the branch rooted at the left child of any vertex with two children is at least as large as the branch rooted at the right child.

Let $v$ and $w$ be two arbitrary non-leaves such that $w$ comes before $v$ in the canonical order. We claim that the branch rooted at $w$ is at least as large as the branch rooted at $v$. Consider first the case that the two vertices do not lie at the same level. The claim is trivial if $w$ is an ancestor of $v$, in particular if it is the root. Otherwise, it follows from part (ii) of Lemma~\ref{Lem:GenComp}, applied to the parent vertices $v'$ and $w'$ of $v$ and $w$ respectively if $w'$ is not an ancestor of $v'$, and by part (i) of Lemma~\ref{Lem:Form2} applied to the branch rooted at $w'$ otherwise.

For vertices at the same level, we prove the claim by induction on the level. If $v$ and $w$ are both at level $1$, i.e., children of the root, then it follows directly from the way $T$ is drawn, so we can focus on the induction step. If $v$ and $w$ have the same parent, then the statement is again a consequence of the way the tree is drawn. Otherwise, let $v'$ and $w'$ be the parents of $v$ and $w$ respectively, and note that $w'$ comes before $v'$ in the canonical order, so the branch rooted at $w'$ is at least as large as the branch rooted at $v'$. Now the statement easily follows from part (i) of Lemma~\ref{Lem:GenComp}, applied to $v'$ and $w'$. 

Let $T'$ be obtained from $T$ by removing all leaves. We know that this is a binary tree, where all vertices are either leaves or have two children. It has $m-1$ internal vertices and  $m$ leaves, see Lemma~\ref{Lem:main}. In view of part (ii) of Lemma~\ref{Lem:Compl}, the leaves of $T'$ lie at the last two levels (since they are exactly the parents of leaves of $T$). So the $i$-th level contains $2^i$ vertices for every $i$, except possibly for the last level. This implies that the height of $T'$ is exactly $h$, and that the number of leaves at the last ($h$-th) level is precisely $2m-2^h$.

Next we show that the leaves of $T'$ at level $h-1$ are precisely the rightmost $2^h-m$ leaves, while the $m-2^{h-1}$ leftmost vertices at this level have two children each. If not, then there are two vertices $v$ and $w$ at level $h-1$ of $T'$ such that $w$ comes before $v$ in the canonical order, but $w$ is a leaf of $T'$ while $v$ is not. Then the branch rooted at $w$ has at most order $5$ in $T$ (all children of $w$ are leaves, and there cannot be more than four), while the branch rooted at $v$ has order at least $7$ in $T$ (the two children are not leaves of $T$ and thus have at least two more children each). But this contradicts the earlier observation that the branch rooted at $w$ cannot be smaller than that rooted at $v$.

Now $T'$ is uniquely characterised. If $n = 5m-1$, then this also already determines $T$: it is obtained from $T'$ by attaching three leaves to every leaf of $T'$, and the shape is as described in the statement of the theorem.

In the other cases, only the position of the branches of order $3$ or $5$ needs to be determined. But this position is again unique by the observation that the branch orders need to decrease according to the canonical order: thus branches of order $5$ have to come before branches of order $4$, and branches of order $3$ after those of order $4$. In each case, the result is exactly the tree that is described.

Finally, we need to consider minimal trees that are not in standard form. Such a tree has a branch of order $7$ (we already know that there can only be at most one: recall part (iii) of Lemma~\ref{Lem:Br_3_5} as well as Lemma~\ref{Lem:main}) whose shape is $[[\bullet,\bullet,\bullet],\bullet,\bullet]$. Replacing this branch by $[[\bullet,\bullet],[\bullet,\bullet]]$, we obtain a tree in standard form (which thus has the stated shape), and vice versa. This completes the proof.
\end{proof}

\section{Asymptotics}

Now that we have characterised the extremal trees, we can also obtain further information on the minimum values of $\I(T)$. The first twenty values are shown in Table~\ref{Tab:Min_Values}. In this section, we will show that there exists a constant $\alpha$ such that
$$m_n = \min_{|T| = n} \I(T) = \Theta(\alpha^n).$$
To this end, we first study the important special case where the number of vertices is of the form $n = 5 \cdot 2^{k-1} - 1$.  In view of Theorem~\ref{thm:main}, the unique minimal tree in this case has height $k$, all vertices at levels $0$ to $k-2$ have precisely two children, while all vertices at level $k-1$ have precisely three children. Thus its two branches are identical, with the number of vertices of the same form:
$$T_{5 \cdot 2^{k}-1} = [T_{5 \cdot 2^{k-1}-1},T_{5 \cdot 2^{k-1}-1}].$$
Define a sequence $x_k$ by $x_k = \J(T_{5 \cdot 2^{k-1}-1}) = \I(T_{5 \cdot 2^{k-1}-1}) + 2$. Note that $x_1 = \J(T_4) = 13$, and the recursion~\eqref{eq:Jrec} for $\J$ shows that
$$x_{k+1} = \J(T_{5 \cdot 2^{k}-1}) = \J(T_{5 \cdot 2^{k-1}-1})^2 - 1 = x_k^2-1.$$
There is a standard technique for treating recursions of this kind, see \cite{aho_1973_doubly}. We take the logarithm to obtain
$$\log x_{k+1} = 2 \log x_k + \log \big( 1 - x_k^{-2} \big),$$
thus by iteration
$$\log x_k = 2^{k-1} \log x_1 + \sum_{j=1}^{k-1} 2^{k-j-1} \log \big( 1 - x_j^{-2} \big).$$
Extending the sum to an infinite series yields
\begin{align}
\log x_k &= 2^{k-1} \log x_1 + \sum_{j=1}^{\infty} 2^{k-j-1} \log \big( 1 - x_j^{-2} \big) - \sum_{j=k}^{\infty} 2^{k-j-1} \log \big( 1 - x_j^{-2} \big) \nonumber \\
&= 2^k \bigg( \frac12 \log x_1 + \sum_{j=1}^{\infty} 2^{-j-1} \log \big( 1 - x_j^{-2} \big) \bigg) - \sum_{i=0}^{\infty} 2^{-i-1} \log \big( 1 - x_{k+i}^{-2} \big).\label{eq:sumrep}
\end{align}
The series converges since it is dominated by a geometric series. The expression inside the large bracket is a constant (note that it does not depend on $k$), which we denote by $\beta$:
\begin{equation}\label{eq:beta_defi}
\beta = \frac12 \log x_1 + \sum_{j=1}^{\infty} 2^{-j-1} \log \big( 1 - x_j^{-2} \big).
\end{equation}
The final sum in~\eqref{eq:sumrep} is positive, so $\log x_k \geq 2^k \beta$. For our purposes, we need somewhat stronger estimates, though: set $X_k = \exp(2^k \beta)$, and notice that
\begin{align*}
\log x_k &= \log X_k - \sum_{i=0}^{\infty} 2^{-i-1} \log \big( 1 - x_{k+i}^{-2} \big) \\
&\leq \log X_k - \sum_{i=0}^{\infty} 2^{-i-1} \log \big( 1 - x_{k}^{-2} \big) \\
&= \log X_k - \log \big( 1 - x_{k}^{-2} \big),
\end{align*}
which gives us
$$x_k(1-x_k^{-2}) \leq X_k.$$
The function $f_1(x) = x(1-x^{-2})$ is increasing on the positive reals, so this implies that
\begin{equation}\label{eq:upper_bd}
x_k \leq f_1^{-1}(X_k) = \frac{X_k + \sqrt{4+X_k^2}}{2} \leq X_k + \frac{1}{X_k}
\end{equation}
(the last inequality is easily verified). On the other hand,~\eqref{eq:sumrep} yields
\begin{align*}
\log x_k &\geq \log X_k - \frac12 \log \big( 1 - x_{k}^{-2} \big) - \frac14 \log \big( 1 - x_{k+1}^{-2} \big) \\
&= \log X_k - \frac12 \log \big( 1 - x_{k}^{-2} \big) - \frac14 \log \big( 1 - (x_k^2-1)^{-2} \big),
\end{align*}
so
$$x_k(1-x_k^{-2})^{1/2} \big(1-(x_k^2-1)^{-2}\big)^{1/4} = \big( (x_k^2-1)^2 - 1 \big)^{1/4} \geq X_k.$$
The function $f_2(x) = \big( (x^2-1)^2 - 1 \big)^{1/4}$ is also increasing, so
\begin{equation}\label{eq:lower_bd}
x_k \geq f_2^{-1}(X_k) = \sqrt{1+\sqrt{1+X_k^4}} \geq X_k + \frac{1}{2X_k}
\end{equation}
(again, the last inequality is straightforward). Now we are ready for the main result of this section.

\begin{thm}\label{thm:asy}
Let $\alpha \approx 1.66928\,37234\,96921\,49740\,26178$ be the constant given by $\alpha = e^{2\beta/5}$, where $\beta$ is defined as in~\eqref{eq:beta_defi}. The minimum number $m_n$ of \icss{} in a rooted tree with $n$ vertices satisfies the inequalities
$$m_n \geq \lceil \alpha^{n+1} \rceil - 2$$
for all $n \geq 1$ (with equality for infinitely many values of $n$) and, for all $n \geq 8$,
$$m_n \leq \lfloor 1.05 \cdot  \alpha^{n+1} \rfloor - 4.$$
\end{thm}
\begin{proof}
We prove both statements by induction on $n$, starting with the lower bound. In fact, we prove the slightly stronger statement that
$$m_n +2 \geq \alpha^{n+1} + \frac{1}{2 \alpha^{n+1}}.$$
Note that if $n = 5 \cdot 2^k - 1$, this is exactly~\eqref{eq:lower_bd}. The  inequality is also easily verified for small $n$ ($n \leq 5$ suffices).  For all $n > 5$, there is a minimal tree $T$ whose root degree is $2$. Let $T_1$ and $T_2$ be its root branches, which are again minimal trees. Let their orders be $n_1$ and $n_2$ respectively ($n_1 + n_2 = n-1$). We have
$$m_n + 2 = \J(T) = \J(T_1)\J(T_2) - 1 = (m_{n_1}+2)(m_{n_2}+2) - 1.$$
By the induction hypothesis, this gives us
\begin{align*}
m_n + 2 &\geq \Big( \alpha^{n_1+1} + \frac{1}{2\alpha^{n_1+1}} \Big)  \cdot \Big( \alpha^{n_2+1} + \frac{1}{2\alpha^{n_2+1}} \Big) - 1 \\
&= \alpha^{n+1} + \frac12 \Big( \alpha^{n_1-n_2} + \alpha^{n_2-n_1} \Big) + \frac{1}{4 \alpha^{n+1}} - 1.
\end{align*}
If $n = 5 \cdot 2^{k-1} - 1$ for some $k$, then we already know that the desired inequality holds (as can be seen from~\eqref{eq:lower_bd}), so we can neglect this case. Otherwise, the characterisation of minimal trees in Theorem~\ref{thm:main} shows that $n_1 \neq n_2$ and thus $|n_1-n_2| \geq 1$. Since $g(x) = \frac12(\alpha^x + \alpha^{-x})$ is an even convex function, this means that
$$m_n+2 \geq \alpha^{n+1} + \frac12 (\alpha + \alpha^{-1}) + \frac{1}{4 \alpha^{n+1}} - 1,$$
and since $\frac12(\alpha + \alpha^{-1}) - 1 > \frac{1}{4\alpha^2} > \frac{1}{4 \alpha^{n+1}}$, the desired inequality follows.

\bigskip

Now we prove the upper bound by another induction. It is easily checked to be valid for $8 \leq n \leq 17$. Assume now that $n \geq 18$, let $k$ be the unique positive integer such that $15 \cdot 2^{k-1} - 1 \leq n < 15 \cdot 2^{k} - 1$, and set $n_1 = 5 \cdot 2^k-1$ and $n_2 = n-n_1-1$. Note that this choice implies $n_2 \leq 2n_1$ as well as $n_1 \geq 9$ and $n_2 \geq 8$. Let $T_1$ and $T_2$ be minimal trees with $n_1$ and $n_2$ vertices respectively, and let $T = [T_1,T_2]$ be the tree whose root branches are $T_1$ and $T_2$ respectively. We can apply~\eqref{eq:upper_bd} to $T_1$ because of the choice of $n_1$, giving us
$$\J(T_1) \leq \alpha^{n_1+1} + \frac{1}{\alpha^{n_1+1}}.$$
Moreover, applying the induction hypothesis to $n_2$, we get
$$\J(T_2) \leq 1.05 \cdot \alpha^{n_2+1} - 2.$$
Hence we have
\begin{align*}
m_n + 2 &\leq \J(T) = \J(T_1)\J(T_2) - 1 \\
&\leq \Big(\alpha^{n_1+1} + \frac{1}{\alpha^{n_1+1}} \Big) \Big(1.05 \cdot \alpha^{n_2+1} -2 \Big) - 1 \\
&= 1.05 \cdot \alpha^{n+1} + 1.05 \cdot \alpha^{n_2-n_1} - 2 \alpha^{n_1+1} - 2\alpha^{-n_1-1} - 1 \\
&\leq 1.05 \cdot \alpha^{n+1} + 1.05 \cdot \alpha^{n_1} - 2 \alpha^{n_1+1} - 1 \\
&= 1.05 \cdot \alpha^{n+1} -(2\alpha-1.05) \alpha^{n_1} - 1 \\
&\leq 1.05 \cdot \alpha^{n+1} -(2\alpha-1.05) \alpha^{9} - 1 \\
&\leq 1.05 \cdot \alpha^{n+1} -2.
\end{align*}
This completes the induction and thus also the proof of the desired upper bound.
\end{proof}

\begin{rem}
Since the series~\eqref{eq:beta_defi} converges rapidly, it is possible to determine the constant $\alpha$ with very high accuracy. However, there is probably no simple expression for $\alpha$. It follows from general results on polynomial recursions that $\alpha$ is irrational \cite{wagner_2021_irrationality} and in fact even transcendental \cite{dubickas_2021_transcendency}.
\end{rem}

\begin{rem}
The constant $1.05$ in the upper bound of Theorem~\ref{thm:asy} is not best possible, but it is not difficult to prove that
$$\limsup_{n \to \infty} \alpha^{-n-1} m_n > 1$$
by considering special sequences (for example $n = 5 \cdot 2^k$). Numerically, the limit superior appears to be $1.0468049642$ (see also Figure~\ref{fig:sequenceplot} for a plot of the sequence given by $\alpha^{-n-1}m_n$). Using the techniques described in~\cite{heuberger_2010_asymptotics}, it might be possible to provide even more precise information on the values of $m_n$.
\end{rem}
\begin{figure}[htbp]
\includegraphics{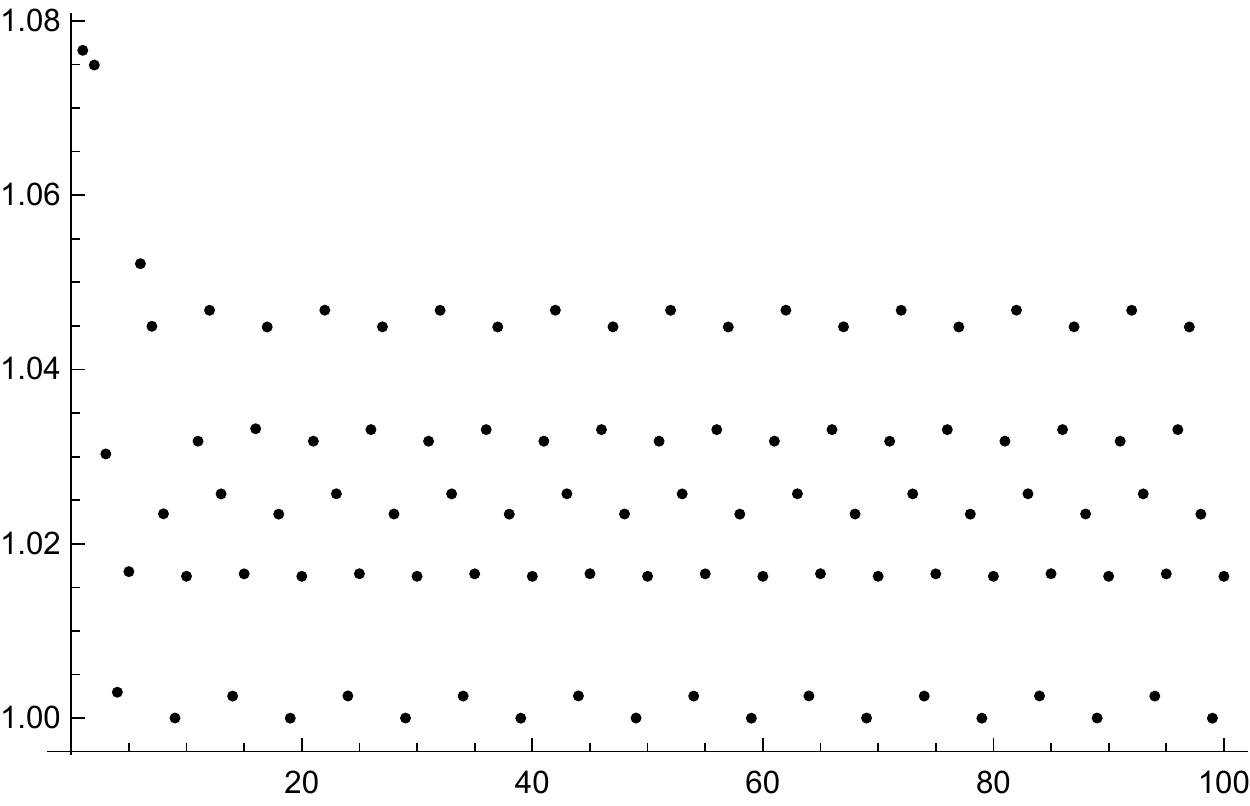}
\caption{The values of $\alpha^{-n-1}m_n$.}\label{fig:sequenceplot}
\end{figure}

\bibliographystyle{abbrv}
\bibliography{Infima_Closed_Sets}

\end{document}